\documentclass[11pt,a4paper]{amsart}
\usepackage{amsmath, amsthm, amsfonts, amssymb}
\usepackage[bookmarks=false]{hyperref}

\newtheorem{letterthm}{Theorem}

\newtheorem{lettercor}[letterthm]{Corollary}

\newtheorem{thm}{Theorem}[section]
\newtheorem{lem}[thm]{Lemma}

\newtheorem{prop}[thm]{Proposition}

\theoremstyle{definition}

\newtheorem{claim}{Claim}
\newtheorem*{newclaim}{Claim}

\newcommand{\R}{\mathbf{R}}
\newcommand{\C}{\mathbf{C}}

\newcommand{\F}{\mathbf{F}}

\newcommand{\N}{\mathbf{N}}

\newcommand{\rL}{\mathord{\text{\rm L}}}
\newcommand{\Ball}{\mathord{\text{\rm Ball}}}
\newcommand{\spn}{\mathord{\text{\rm span}}}

\begin{document}

\title[Gamma stability in free product von Neumann algebras]{Gamma stability in free product von Neumann algebras}

\begin{abstract}
Let $(M, \varphi) = (M_1, \varphi_1) \ast (M_2, \varphi_2)$ be a free product of arbitrary von Neumann algebras endowed with faithful normal states. Assume that the centralizer $M_1^{\varphi_1}$ is diffuse. We first show that any intermediate subalgebra $M_1 \subset Q \subset M$ which has nontrivial central sequences in $M$ is necessarily equal to $M_1$. Then we obtain a general structural result for all the  intermediate subalgebras $M_1 \subset Q \subset M$ with expectation. We deduce that any diffuse amenable von Neumann algebra can be concretely realized as a maximal amenable subalgebra with expectation inside a full nonamenable type ${\rm III_1}$ factor. This provides the first class of concrete maximal amenable subalgebras in the framework of type ${\rm III}$ factors. We finally strengthen all these results in the case of tracial free product von Neumann algebras. 
\end{abstract}

\author{Cyril Houdayer}
\address{CNRS - Universit\'e Paris-Est - Marne-la-Vall\'ee \\
    LAMA UMR 8050 \\ 77454 Marne-la-Vall\'ee cedex~2 
\\ France}
\email{cyril.houdayer@u-pem.fr}

\thanks{Research supported by ANR grant NEUMANN}



\maketitle

\section{Introduction and statement of the main results}

A von Neumann algebra $M \subset \mathbf B(H)$ (with separable predual) is {\em amenable} if there exists a norm one projection $E : \mathbf B(H) \to M$. By Connes' celebrated result \cite{Co75b}, all the amenable von Neumann algebras are {\em hyperfinite}. Moreover, the amenable or hyperfinite factors are completely classified by their flows of weights (see \cite{Co72, Co75b, Co85, Ha84}). In particular, there is a unique amenable ${\rm II_1}$ factor \cite{Co75b}: it is the hyperfinite ${\rm II_1}$ factor of Murray and von Neumann \cite{MvN43}.

Since the amenable von Neumann algebras form a monotone class, any von Neumann algebra admits maximal amenable subalgebras. The first concrete examples of maximal amenable subalgebras inside ${\rm II_1}$ factors were obtained by Popa in \cite{Po83}. He showed that any {\em generator} masa $A$ in a free group factor $\rL(\F_n)$ with $n \geq 2$ is maximal amenable. This result answered in the negative a question raised by Kadison. Indeed, $A \subset \rL(\F_n)$ is an abelian subalgebra generated by a selfadjoint operator and yet there is no intermediate hyperfinite subfactor in $\rL(\F_n)$ which contains $A$ as a subalgebra. Popa discovered in \cite{Po83} a powerful method to prove that a given amenable subalgebra is maximal amenable inside an ambient ${\rm II_1}$ factor. Using this strategy for the generator masa $A \subset \rL(\F_n)$, he first showed that $A$ satisfies a certain {\em asymptotic orthogonality property} and then deduced that $A$ is maximal amenable in $\rL(\F_n)$ using various {\em mixing} techniques. His results actually showed that the generator masa $A$ is maximal Gamma inside $\rL(\F_n)$. Recall that a ${\rm II_1}$ factor $M$ (with separable predual) has {\em property Gamma} of Murray and von Neumann \cite{MvN43} if there exists a sequence of unitaries $u_n \in \mathcal U(M)$ such that $\lim_{n \to \infty} \tau(u_n) = 0$ and $\lim_{n \to \infty} \|x u_n - u_n x\|_2 = 0$ for all $x \in M$. 

Subsequently, Cameron, Fang, Ravichandran and White proved in \cite{CFRW08} that the {\em radial} masa in a free group factor $\rL(\F_n)$ with $2 \leq n < \infty$ is maximal amenable. Recently, the author vastly generalized in \cite{Ho12a, Ho12b} Popa's results from \cite{Po83} and obtained many new examples of maximal amenable subalgebras inside the crossed product ${\rm II_1}$ factors associated with free Bogoljubov actions of amenable groups. Very recently, Boutonnet and Carderi showed in \cite{BC13} that any infinite maximal amenable subgroup $\Lambda$ in a Gromov word-hyperbolic group $\Gamma$ gives rise to a maximal amenable subalgebra $\rL(\Lambda)$ inside the group von Neumann algebra $\rL(\Gamma)$. For other related results regarding maximal amenability in the framework of ${\rm II_1}$ factors, we refer the reader to \cite{Br12, Fa06, Ga09, Ge95, Jo10, Po13, Sh05}.

In this paper, we obtain new results regarding maximal amenability and Gamma stability for subalgebras of free products of {\em arbitrary} von Neumann algebras. We will be particularly interested in the structure of free product type ${\rm III}$ factors. Before stating our main results, we first introduce some terminology. Recall that a von Neumann algebra $M$ is {\em diffuse} if $M$ has no minimal projection. We say that a von Neumann subalgebra $Q \subset M$ is {\em with expectation} if there exists a faithful normal conditional expectation $E_Q : M \to Q$. Let now $\omega \in \beta(\N) \setminus \N$ be a non-principal ultrafilter. We say that a von Neumann algebra $M$ has {\em property Gamma} if the central sequence algebra $M' \cap M^\omega$ is diffuse. Observe that in the case when $M$ is a ${\rm II_1}$ factor with separable predual, this definition is equivalent to the property Gamma of Murray and von Neumann \cite{MvN43} (see e.g.\ \cite[Corollary 3.8]{Co74}).

Our first main result deals with Gamma stability inside arbitrary free product von Neumann algebras $(M, \varphi) = (M_1, \varphi_1) \ast (M_2, \varphi_2)$. We show in Theorem \ref{thmA} that the subalgebra $M_1 \subset M$ sits in a very rigid position with respect to taking central sequences inside $M$.

\begin{letterthm}\label{thmA}
Let $(M_1, \varphi_1)$ and $(M_2, \varphi_2)$ be $\sigma$-finite von Neumann algebras endowed with faithful normal states. Assume that the centralizer $M_1^{\varphi_1}$ is diffuse. Denote by $(M, \varphi) = (M_1, \varphi_1) \ast (M_2, \varphi_2)$ the free product von Neumann algebra. 

Then the inclusion $M_1 \subset M$ is {\em Gamma stable} in the following sense:  for every intermediate von Neumann subalgebra $M_1 \subset Q \subset M$ such that $Q' \cap M^\omega$ is diffuse, we have $Q = M_1$.
\end{letterthm}

It is worth noticing that in the statement of Theorem \ref{thmA}, the intermediate subalgebra $M_1 \subset Q \subset M$ is not assumed {\em a priori} to be with expectation in $M$. The proof of Theorem \ref{thmA} is based on a key result (see Theorem \ref{key-thm}) which is a generalization of Popa's result \cite[Lemma 2.1]{Po83} regarding asymptotic orthogonality for free group factors to arbitrary free product von Neumann algebras. The proof uses Popa's original method together with $\varepsilon$-orthogonality techniques from \cite{Ho12a, Ho12b}.

In order to obtain structural results for the intermediate subalgebras $M_1 \subset Q \subset M$, we will next assume that $Q$ is with expectation in $M$ in the statement of Corollary \ref{corB}. Recall that a factor $M$ (with separable predual) is {\em full} if its asymptotic centralizer $M_\omega$ is trivial (see \cite{Co74}). Observe that by \cite[Theorem 5.3]{AH12}, this is equivalent to $M' \cap M^\omega = \C$.

\begin{lettercor}\label{corB}
Let $(M_1, \varphi_1)$ and $(M_2, \varphi_2)$ be von Neumann algebras with separable predual endowed with faithful normal states. Assume that the centralizer $M_1^{\varphi_1}$ is diffuse. Denote by $(M, \varphi) = (M_1, \varphi_1) \ast (M_2, \varphi_2)$ the free product von Neumann algebra. 

Then any intermediate von Neumann subalgebra $M_1 \subset Q \subset M$ with faithful normal conditional expectation $E_Q : M \to Q$ is globally invariant under the modular automorphism group $(\sigma_t^\varphi)$. Moreover, there exists a sequence of pairwise orthogonal projections $z_n \in Q' \cap M \subset \mathcal Z(M_1)$ such that $\sum_{n} z_n = 1$ and
\begin{itemize}
\item $M_1z_0 = Qz_0$ and
\item $Q z_n$ is a full nonamenable factor such that $(Qz_n)' \cap (z_n Mz_n)^\omega = \C z_n$ for every $n \geq 1$.
\end{itemize}
\end{lettercor}

Corollary \ref{corB} generalizes and strengthens \cite[Lemma 3.1]{Po83} and \cite[Lemma 4.4]{Ge95}. Corollary \ref{corB} moreover implies that if $M_1$ has property Gamma, then $M_1 \subset M$ is a maximal Gamma subalgebra with expectation in $M$. The structural result in Corollary \ref{corB} allows us to obtain a wide range of maximal amenable subalgebras inside nonamenable factors. In particular, Corollary \ref{corC} below provides the first class of concrete maximal amenable subalgebras with expectation in the framework of type ${\rm III}$ factors.

\begin{lettercor}\label{corC}
Any diffuse amenable von Neumann algebra with separable predual can be concretely realized as a maximal amenable subalgebra with expectation inside a full nonamenable type ${\rm III_1}$ factor.
\end{lettercor}

Our main last result deals with Gamma stability for subalgebras of {\em tracial} free product von Neumann algebras. Theorem \ref{thmD} below is a further generalization of Corollary \ref{corB} where the subalgebra $Q \subset M$ is only assumed to have a diffuse intersection with $M_1$.

\begin{letterthm}\label{thmD}
Let $(M_1, \tau_1)$ and $(M_2, \tau_2)$ be von Neumann algebras with separable predual endowed with faithful normal tracial states. Assume that $M_1$ is diffuse. Denote by $(M, \tau) = (M_1, \tau_1) \ast (M_2, \tau_2)$ the tracial free product von Neumann algebra. 

Then for every von Neumann subalgebra $Q \subset M$ such that $Q \cap M_1$ is diffuse, there exists a central projection $z \in \mathcal Z(Q' \cap M) \cap \mathcal Z(Q' \cap M^\omega) \subset M_1$ such that 
\begin{itemize}
\item $Q z \subset z M_1 z$ and 
\item $(Q' \cap M^\omega)(1 - z) = (Q' \cap M) (1 - z)$ is discrete.
\end{itemize}
\end{letterthm}

Theorem \ref{thmD} shows in particular that whenever $Q \subset M$ is a subalgebra such that both $Q \cap M_1$ and $Q' \cap M^\omega$ are diffuse, then $Q \subset M_1$ (see Theorem \ref{step1}). This is a strengthening of the Gamma stability result in Theorem \ref{thmA}. Besides the asymptotic orthogonality property obtained in Theorem \ref{key-thm}, the proof of Theorem \ref{thmD} uses two more ingredients of ${\rm II_1}$ factors: Popa's intertwining techniques \cite{Po01, Po03} and Peterson's $\rL^2$-rigidity results for tracial free product von Neumann algebras \cite{Pe06}.

In Section \ref{preliminaries}, we recall a few preliminaires on free product and ultraproduct von Neumann algebras. In Section \ref{asymptotic}, we prove the key result regarding asymptotic orthogonality inside free products of arbitrary von Neumann algebras. Finally, we prove in Section \ref{proofs} the main results of the paper.

\subsection*{Acknowledgments}

I am grateful to R\'emi Boutonnet and Sven Raum for their valuable comments regarding a first draft of this manuscript. I especially thank R\'emi for pointing out a gap in the initial proof of Proposition \ref{gamma}. Finally, I thank the referee for carefully reading the paper and useful remarks.

\section{Preliminaries}\label{preliminaries}

We fix once and for all a non-principal ultrafilter $\omega \in \beta(\N) \setminus \N$. All the von Neumann algebras that we consider in this paper are assumed to be $\sigma$-{\em finite}, that is, countably decomposable. We say that $M$ is a {\em tracial} von Neumann algebra if $M$ admits a faithful normal tracial state $\tau$.

\subsection*{Background on $\sigma$-finite von Neumann algebras}

Let $M$ be any $\sigma$-finite von Neumann algebra. We denote by $\Ball(M)$ the unit ball of $M$ with respect to the uniform norm $\|\cdot\|_\infty$, $\mathcal U(M)$ the group of unitaries in $M$ and $\mathcal Z(M)$ the center of $M$. Let $\varphi \in M_\ast$ be a faithful normal state. We denote by $\rL^2(M, \varphi)$ (or simply $\rL^2(M)$ when no confusion is possible) the GNS $\rL^2$-completion of $M$ with respect to the inner product defined by $\langle x, y \rangle_\varphi = \varphi(y^*x)$ for all $x, y \in M$. We  denote by $\Lambda_\varphi : M \to \rL^2(M) : x \mapsto \Lambda_\varphi(x)$ the canonical embedding and by $J_\varphi : \rL^2(M) \to \rL^2(M)$ the canonical conjugation. We have $x \Lambda_\varphi (y) = \Lambda_\varphi(xy)$ for all $x, y \in M$. 

We say that two elements $x, y \in M$ are $\varphi$-{\em orthogonal} in $M$ if $\varphi(y^* x) = 0$ or equivalently if the vectors $\Lambda_\varphi(x)$ and $\Lambda_\varphi(y)$ are orthogonal in the Hilbert space $\rL^2(M)$. For all $x \in M$, write $\|x\|_\varphi = \varphi(x^*x)^{1/2}$ and $\|x\|_\varphi^\sharp = \varphi(x^*x + xx^*)^{1/2}$. Recall that the strong (resp.\ $\ast$-strong) topology on uniformly bounded subsets of $M$ coincides with the topology defined by $\|\cdot\|_\varphi$ (resp.\ $\|\cdot\|_\varphi^\sharp$).

An element $x \in M$ is said to be {\em analytic} with respect to the modular automorphism group $(\sigma_t^\varphi)$ if the function $\R \to M : t \mapsto \sigma_t^\varphi(x)$ can be extended to an $M$-valued entire analytic function over $\C$.

We will be using the following standard facts.

\begin{prop}\label{modular-analytic}
Let $(M, \varphi)$ be any $\sigma$-finite von Neumann algebra endowed with a faithful normal state. 
\begin{enumerate}
\item The subset $\mathcal A \subset M$ of all the elements in $M$ which are analytic with respect to the modular automorphism group $(\sigma_t^\varphi)$ forms a unital $\sigma$-strongly dense $\ast$-subalgebra of $M$.
\item For all $a \in \mathcal A$ and all $x \in M$, we have 
$$\Lambda_\varphi(x a) = J_\varphi \sigma_{-{\rm i}/2}^\varphi(a^*) J_\varphi \, \Lambda_\varphi(x).$$
\item For all $a \in \mathcal A$ and all $x \in M$, we have 
$$\varphi(a x) = \varphi(x \sigma_{-\rm i}^\varphi(a)).$$
In particular, for all $a \in \mathcal A$ and all $x, y \in M$, we have that $xa$ and $y$ are $\varphi$-orthogonal in $M$ if and only if $x$ and $y \sigma_{{\rm i}}^\varphi(a)^*$ are $\varphi$-orthogonal in $M$.
\end{enumerate}
\end{prop}

\begin{proof}
$(1)$ follows from \cite[Lemma VIII.2.3]{Ta03} and $(2)$ follows from \cite[Lemma VIII.3.10]{Ta03}. Let us prove $(3)$. For every $a \in \mathcal A$ and every $x \in M$, we have
\begin{align*}
\varphi(x \sigma_{-\rm i}^\varphi(a)) &= \langle \Lambda_\varphi(x \sigma_{-\rm i}^\varphi(a)), \Lambda_\varphi(1) \rangle_\varphi \\
& = \langle J_\varphi \sigma_{\rm i/2}^\varphi(a^*) J_\varphi \,  \Lambda_\varphi(x), \Lambda_\varphi(1) \rangle_\varphi \\
& = \langle   \Lambda_\varphi(x), J_\varphi \sigma_{\rm -i/2}^\varphi(a) J_\varphi \, \Lambda_\varphi(1) \rangle_\varphi \\
& = \langle  \Lambda_\varphi(x), \Lambda_\varphi(a^*) \rangle_\varphi \\
& = \varphi(a x).
\end{align*}
In particular, for all $a \in \mathcal A$ and all $x, y \in M$, we have
$$\varphi((y \sigma_{\rm i}^\varphi(a)^*)^* \, x ) = \varphi(\sigma_{\rm i}^\varphi(a) \, y^* x) = \varphi(y^* \, xa).$$
Hence $xa$ and $y$ are $\varphi$-orthogonal in $M$ if and only if $x$ and $y \sigma_{{\rm i}}^\varphi(a)^*$ are $\varphi$-orthogonal in $M$.
\end{proof}

\begin{prop}\label{diffuse-centralizer}
Let $M$ be any $\sigma$-finite von Neumann algebra.
\begin{enumerate}
\item We have that $M$ is diffuse if and only if there exists a faithful normal state $\varphi \in M_*$ such that the centralizer $M^\varphi$ is diffuse. Moreover in that case, there exists a unitary $u \in \mathcal U(M^\varphi)$ such that $u^k \to 0$ weakly as $|k| \to \infty$.
\item Let $N \subset M$ be a von Neumann subalgebra with expectation. If $N$ is diffuse, so is $M$.
\end{enumerate}
\end{prop}

\begin{proof}
$(1)$ Assume first that $M$ is diffuse. There exists a sequence of pairwise orthogonal projections $z_n \in \mathcal Z(M)$ such that $\sum_n z_n = 1$, $M z_0$ is a von Neumann algebra with a diffuse center and $M z_n$ is a diffuse factor for every $n \geq 1$. Choose any faithful normal state $\varphi_0$ on $Mz_0$. By \cite[Theorem 11.1]{HS90}, for every $n \geq 1$, choose a faithful normal state $\varphi_n$ on $Mz_n$ such that the centralizer $(M z_n)^{\varphi_n}$ is diffuse. Let $(a_n)_{n}$ be a sequence of positive reals so that $\sum_{n} a_n = 1$. The formula $\varphi = \sum_{n} a_n \varphi_n$ defines a faithful normal state on $M$ such that 
$$M^\varphi = \bigoplus_{n} (M z_n)^{\varphi_n}.$$
Therefore, $M^\varphi$ is diffuse. 

Assume next that $M^\varphi$ is diffuse for some faithful normal state $\varphi \in M_\ast$. Using the above decomposition, for every $n \geq 1$ such that $z_n \neq 0$, letting $\varphi_n = \frac{1}{\varphi(z_n)}\varphi(\cdot z_n)$, we have that $(M z_n)^{\varphi_n} = M^\varphi z_n$ is diffuse. Therefore $M z_n$ is a non-type ${\rm I}$ factor and so is diffuse. Thus, $M$ is diffuse.

When $M^\varphi$ is diffuse, take $A \subset M^\varphi$ a maximal abelian subalgebra. Then $A$ is necessarily diffuse. Then choose a diffuse subalgebra $B \subset A$ with separable predual. Since $B \cong \rL^\infty(\mathbf T)$, we can then take a unitary $u \in \mathcal U(B)$ such that $u^k \to 0$ weakly as $|k| \to \infty$. 

$(2)$ Denote by $E : M \to N$ a faithful normal conditional expectation and choose a faithful normal state $\psi \in N_\ast$ such that $N^\psi$ is diffuse. Then $\varphi = \psi \circ E$ is a faithful normal state on $M$ such that $N^\psi \subset M^\varphi$. Since $N^\psi$ is diffuse and $M^\varphi$ is tracial, $M^\varphi$ is diffuse and so is $M$ by item $(1)$ of the proposition.
\end{proof}

\subsection*{Free product von Neumann algebras}

For $i = 1, 2$, let $(M_i, \varphi_i)$ be any $\sigma$-finite von Neumann algebra endowed with a faithful normal state. The {\em free product von Neumann algebra} $(M, \varphi) = (M_1, \varphi_1) \ast (M_2, \varphi_2)$ is the von Neumann algebra $M$ generated by $M_1$ and $M_2$ where the faithful normal state $\varphi$ satisfies the following {\em freeness condition}:
$$\varphi(x_1 \cdots x_n) = 0 \; \text{ whenever } \; x_j \in M_{i_j} \ominus \C \; \text{ and } \; i_1 \neq \cdots \neq  i_{n}.$$
Here and in what follows, we denote by $M_i \ominus \C = \ker(\varphi_i)$. We refer to the product $x_1 \cdots x_n$ where $x_j \in M_{i_j} \ominus \C$ and $i_1 \neq \cdots \neq i_{n}$ as a {\em reduced word} in $(M_{i_1} \ominus \C) \cdots (M_{i_n} \ominus \C)$ of {\em length} $n \geq 1$. The linear span of $1$ and of all the reduced words in $(M_{i_1} \ominus \C) \cdots (M_{i_n} \ominus \C)$ where $n \geq 1$ and $i_1 \neq \cdots \neq i_{n}$ forms a unital $\sigma$-strongly dense $\ast$-subalgebra of $M$.

For all $n \geq 1$ and all $i_1 \neq \cdots \neq i_{n}$, the mapping 
\begin{align*}
\rL^2((M_{i_1} \ominus \C) \cdots (M_{i_n} \ominus \C), \varphi) & \to \rL^2(M_{i_1} \ominus \C, \varphi_{i_1}) \otimes \cdots \otimes \rL^2 (M_{i_n} \ominus \C, \varphi_{i_n}) \\
\Lambda_\varphi(x_1 \cdots x_n) & \mapsto \Lambda_{\varphi_{i_1}}(x_1) \otimes \cdots \otimes \Lambda_{\varphi_{i_n}}(x_n)  
\end{align*}
defines a unitary operator. Moreover, we have
$$\rL^2(M, \varphi) = \C \oplus \bigoplus_{n \geq 1} \bigoplus_{i_1 \neq \cdots \neq i_n} \rL^2(M_{i_1} \ominus \C, \varphi_{i_1}) \otimes \cdots \otimes \rL^2 (M_{i_n} \ominus \C, \varphi_{i_n}).$$

For all $t \in \R$, we have $\sigma_t^\varphi = \sigma_t^{\varphi_1} \ast \sigma_t^{\varphi_2}$ (see \cite[Lemma 1]{Ba93} and \cite[Theorem 1]{Dy92}). By \cite[Theorem IX.4.2]{Ta03}, there exists a unique $\varphi$-preserving faithful normal conditional expectation $E_{M_1} : M \to M_1$. Moreover, we have $E_{M_1}(x_1 \cdots x_n) = 0$ for all the reduced words $x_1 \cdots x_n$ which contains at least one letter from $M_2 \ominus \C$ (see \cite[Lemma 2.1]{Ue11}). For more on free product von Neumann algebras, we refer the reader to \cite{Ue98, Ue11, Vo85, Vo92}.

\subsection*{Ultraproduct von Neumann algebras}

Let $M$ be any $\sigma$-finite von Neumann algebra. Define
\begin{align*}
\mathcal I^\omega(M) &= \left\{ (x_n)_n \in \ell^\infty(\N, M) : x_n \to 0 \ast-\text{strongly as } n \to \omega \right\} \\
\mathcal M^\omega(M) &= \left \{ (x_n)_n \in \ell^\infty(\N, M) :  (x_n)_n \, \mathcal I^\omega(M) \subset \mathcal I^\omega(M) \text{ and } \mathcal I^\omega(M) \, (x_n)_n \subset \mathcal I^\omega(M)\right\}.
\end{align*}

We have that the {\em multiplier algebra} $\mathcal M^\omega(M)$ is a C$^*$-algebra and $\mathcal I^\omega(M) \subset \mathcal M^\omega(M)$ is a norm closed two-sided ideal. Following \cite{Oc85}, we define the {\em ultraproduct von Neumann algebra} $M^\omega$ by $M^\omega = \mathcal M^\omega(M) / \mathcal I^\omega(M)$. We denote the image of $(x_n)_n \in \mathcal M^\omega(M)$ by $(x_n)^\omega \in M^\omega$. 

For all $x \in M$, the constant sequence $(x)_n$ lies in the multiplier algebra $\mathcal M^\omega(M)$. We will then identify $M$ with $(M + \mathcal I^\omega(M))/ \mathcal I^\omega(M)$ and regard $M \subset M^\omega$ as a von Neumann subalgebra. The map $E_M : M^\omega \to M : (x_n)^\omega \mapsto \sigma \text{-weak} \lim_{n \to \omega} x_n$ is a faithful normal conditional expectation. For every faithful normal state $\varphi \in M_\ast$, the formula $\varphi^\omega = \varphi \circ E_M$ defines a faithful normal state on $M^\omega$. Observe that $\varphi^\omega((x_n)^\omega) = \lim_{n \to \omega} \varphi(x_n)$ for all $(x_n)^\omega \in M^\omega$.

Let $Q \subset M$ be any von Neumann subalgebra with faithful normal conditional expectation $E_Q : M \to Q$. Choose a faithful normal state $\varphi$ on $Q$ and still denote by $\varphi$ the faithful normal state $\varphi \circ E_Q$ on $M$. We have $\ell^\infty(\N, Q) \subset \ell^\infty(\N, M)$, $\mathcal I^\omega(Q) \subset \mathcal I^\omega(M)$ and $\mathcal M^\omega(Q) \subset \mathcal M^\omega(M)$. We will then identify $Q^\omega = \mathcal M^\omega(Q) / \mathcal I^\omega(Q)$ with $(\mathcal M^\omega(Q) + \mathcal I^\omega(M)) / \mathcal I^\omega(M)$ and regard $Q^\omega \subset M^\omega$ as a von Neumann subalgebra. Observe that the norm $\|\cdot\|_{(\varphi | Q)^\omega}$ on $Q^\omega$ is the restriction of the norm $\|\cdot\|_{\varphi^\omega}$ to $Q^\omega$. Observe moreover that $(E_Q(x_n))_n \in \mathcal I^\omega(Q)$ for all $(x_n)_n \in \mathcal I^\omega(M)$ and $(E_Q(x_n))_n \in \mathcal M^\omega(Q)$ for all $(x_n)_n \in \mathcal M^\omega(M)$. Therefore, the mapping $E_{Q^\omega} : M^\omega \to Q^\omega : (x_n)^\omega \mapsto (E_Q(x_n))^\omega$ is a well-defined conditional expectation satisfying $\varphi^\omega \circ E_{Q^\omega} = \varphi^\omega$. Hence, $E_{Q^\omega} : M^\omega \to Q^\omega$ is a faithful normal conditional expectation.

Put $\mathcal H = \rL^2(M, \varphi)$. The {\em ultraproduct Hilbert space} $\mathcal H^\omega$ is defined to  be the quotient of $\ell^\infty(\N, \mathcal H)$ by the subspace consisting in sequences $(\xi_n)_n$ satisfying $\lim_{n \to \omega} \|\xi_n\|_{\mathcal H} = 0$. We denote the image of $(\xi_n)_n \in \ell^\infty(\N, \mathcal H)$ by $(\xi_n)_\omega \in \mathcal H^\omega$. The inner product space structure on the Hilbert space $\mathcal H^\omega$ is defined by $\langle (\xi_n)_\omega, (\eta_n)_\omega\rangle_{\mathcal H^\omega} = \lim_{n \to \omega} \langle \xi_n, \eta_n\rangle_{\mathcal H}$. The GNS Hilbert space $\rL^2(M^\omega, \varphi^\omega)$ can be embedded into $\mathcal H^\omega$ as a closed subspace by $\Lambda_{\varphi^\omega}((x_n)^\omega) \mapsto (\Lambda_\varphi(x_n))_\omega$. For more on ultraproduct von Neumann algebras, we refer the reader to \cite{AH12, Oc85}.

Put $x \varphi = \varphi (\cdot x)$ and $\varphi x = \varphi(x \cdot)$ for all $x \in M$ and all $\varphi \in M_\ast$. We will be using the following standard facts.

\begin{lem}\label{lemma-states}
Let $(M, \varphi)$ be any $\sigma$-finite von Neumann algebra endowed with a faithful normal state. Then for every $x \in M$, we have
$$\|x\varphi\| \leq \|x\|_\varphi, \; \|\varphi x\| \leq \|x^*\|_\varphi \text{ and } \|x \varphi - \varphi x\| = \|x^* \varphi - \varphi x^*\|.$$
\end{lem}

\begin{proof}
Let $x \in M$. Using the Cauchy-Schwarz inequality, for all $y \in \Ball(M)$, we have
$$|(x \varphi)(y)| = |\varphi(y x)| \leq \|y^*\|_\varphi \, \|x\|_\varphi \leq \|x\|_\varphi$$
and hence $\|x \varphi\| \leq \|x\|_\varphi$. Likewise, for all $y \in \Ball(M)$, we have
$$|(\varphi x)(y)| = |\varphi(x y)| \leq \|x^*\|_\varphi \, \|y\|_\varphi \leq \|x^*\|_\varphi$$
and hence $\|\varphi x\| \leq \|x^*\|_\varphi$. Moreover, for all $y \in \Ball(M)$, we have
$$|(x^* \varphi - \varphi x^*)(y)| = |\varphi(y x^* - x^* y)| = |\overline{\varphi(y x^* - x^* y)}| = |\varphi(x y^* - y^* x)| = |(x \varphi - \varphi x)(y^*)|.$$
This implies that $\|x \varphi - \varphi x\| = \|x^* \varphi - \varphi x^*\|$.
\end{proof}

\begin{prop}\label{proposition-ultraproduct}
Let $(M, \varphi)$ be any $\sigma$-finite von Neumann algebra endowed with a faithful normal state.
\begin{enumerate}
\item For every $(x_n)_n \in \mathcal M^\omega(M)$ and every $(y_n)_n \in \ell^\infty(\N, M)$ such that $x_n - y_n \to 0$ $\ast$-strongly as $n \to \omega$, we have $(y_n)_n \in \mathcal M^\omega(M)$ and $(x_n)^\omega = (y_n)^\omega \in M^\omega$.
\item For every $(x_n)_n \in \ell^\infty(\N, M)$ satisfying $\lim_{n \to \omega} \|x_n \varphi - \varphi x_n\| = 0$, we have $(x_n)_n \in \mathcal M^\omega(M)$ and $(x_n)^\omega \in (M^\omega)^{\varphi^\omega}$.
\item For every projection $e \in M^\omega$, there exists a sequence of projections $(e_n)_n \in \mathcal M^\omega(M)$ such that $e = (e_n)^\omega$.
\end{enumerate}
\end{prop}

\begin{proof}
$(1)$ Let $(x_n)_n \in \mathcal M^\omega(M)$ and $(y_n)_n \in \ell^\infty(\N, M)$ such that $x_n - y_n \to 0$ $\ast$-strongly as $n \to \omega$. Then $(y_n - x_n)_n \in \mathcal I^\omega(M) \subset \mathcal M^\omega(M)$ and hence $(y_n)_n = (y_n - x_n)_n + (x_n)_n \in \mathcal M^\omega(M)$. Moreover, by the definition of the ultraproduct von Neumann algebra $M^\omega$, we have $(x_n)^\omega = (y_n)^\omega \in M^\omega$.

$(2)$ Let $(x_n)_n \in \ell^\infty(\N, M)$ such that $\lim_{n \to \omega} \|x_n \varphi - \varphi x_n\| = 0$. Let $(b_n)_n \in \mathcal I^\omega(M)$. We may assume that $\max \{\|x_n\|_\infty, \|b_n\|_\infty : n \in \N\} \leq 1$. Using the Cauchy-Schwarz inequality, for all $n \in \N$, we have
\begin{align*}
(\|x_n b_n\|_\varphi^\sharp)^2 &= \varphi(b_n^* \, x_n^* x_n b_n) + \varphi(x_n \, b_n b_n^* x_n^*) \\
& \leq \|b_n\|_\varphi \, \|x_n^* x_n b_n\|_\varphi + |(x_n \varphi - \varphi x_n)(b_n b_n^* x_n^*)| + |\varphi(b_n \, b_n^* x_n^* x_n)| \\
& \leq \|b_n\|_\varphi + \|x_n \varphi - \varphi x_n\| \, \|b_n b_n^* x_n^*\|_\infty + \|b_n^*\|_\varphi \, \|b_n^* x_n^* x_n\|_\varphi \\
& \leq \|b_n\|_\varphi + \|x_n \varphi - \varphi x_n\| + \|b_n^*\|_\varphi.
\end{align*}
Therefore, we obtain $\lim_{n \to \omega} \|x_n b_n\|_\varphi^\sharp = 0$ and so $(x_n b_n)_n \in \mathcal I^\omega(M)$. Likewise, for all $n \in \N$, we have
\begin{align*}
(\|b_n x_n\|_\varphi^\sharp)^2 &= \varphi(x_n^* \, b_n^* b_n x_n) + \varphi(b_n \, x_n x_n^* b_n^*) \\
& \leq |(x_n^* \varphi - \varphi x_n^*)(b_n^* b_n x_n)| + |\varphi(b_n^* \, b_n x_n x_n^*)| + \|b_n^*\|_\varphi \, \|x_n x_n^* b_n^*\|_\varphi \\
& \leq \|x_n^* \varphi - \varphi x_n^*\| \, \|b_n^* b_n x_n\|_\infty+ \|b_n\|_\varphi \, \|b_n x_n x_n^*\|_\varphi + \|b_n^*\|_\varphi \\
& \leq \|x_n \varphi - \varphi x_n\| + \|b_n\|_\varphi + \|b_n^*\|_\varphi.
\end{align*}
Therefore, we obtain $\lim_{n \to \omega} \|b_n x_n\|_\varphi^\sharp = 0$ and so $(b_n x_n)_n \in \mathcal I^\omega(M)$. This shows that $(x_n)_n \in \mathcal M^\omega(M)$. Moreover, $x = (x_n)^\omega \in (M^\omega)^{\varphi^\omega}$ by \cite[Lemma 4.35]{AH12}.

$(3)$ The proof is identical to the one of \cite[Proposition 1.1.3]{Co75a}. Let $e \in M^\omega$ be any projection. We may choose a sequence $(x_n)_n \in \mathcal M^\omega(M)$ such that $\|x_n\|_\infty \leq 1$ for all $n \in \N$ and $e = (x_n)^\omega$. Put $y_n = x_n^* x_n$ for all $n \in \N$. Since $e = e^* e$, we have $\lim_{n \to \omega} \|x_n - y_n\|_\varphi^\sharp = 0$, $(y_n)_n \in \mathcal M^\omega(M)$ and $e = (y_n)^\omega$. Since $e = e^2$, we moreover have $\lim_{n \to \omega} \|y_n - y_n^2\|_\varphi^\sharp = 0$. Put $\varepsilon_n = \|y_n - y_n^2\|_\varphi$. Letting $e_n = \mathbf 1_{[1 - \sqrt{\varepsilon_n}, 1]}(y_n) \in M$ for all $n \in \N$, we have $\lim_{n \to \omega} \| y_n - e_n \|_\varphi^\sharp = 0$ by \cite[Lemma 1.1.5]{Co75a}. It follows that $(e_n)_n \in \mathcal M^\omega(M)$ and $e = (e_n)^\omega \in M^\omega$ by item (1) of the proposition.
\end{proof}

The next proposition will be useful to prove Corollary \ref{corB}.

\begin{prop}\label{gamma}
Let $M$ be any factor with separable predual and $Q \subset M$ any irreducible subfactor with expectation. Then, either $Q' \cap M^\omega = \C$ or $Q' \cap M^\omega$ is diffuse.
\end{prop}

\begin{proof}
Denote by $E_Q : M \to Q$ the faithful normal conditional expectation. Choose a faithful normal state on $Q$ and still denote by $\varphi$ the faithful normal state $\varphi \circ E_Q$ on $M$. Since $Q$ is globally invariant under the modular automorphism group $(\sigma_t^\varphi)$ and since $\sigma_t^{\varphi^\omega}(x) = \sigma_t^\varphi(x)$ for all $x \in M$, the relative commutant $Q' \cap M^\omega$ is globally invariant under the modular automorphism group $(\sigma_t^{\varphi^\omega})$. Hence $(Q' \cap M^\omega)^{\varphi^\omega} = (Q' \cap M^\omega) \cap (M^\omega)^{\varphi^\omega} = Q' \cap (M^\omega)^{\varphi^\omega} $.

\begin{newclaim}
Either $Q' \cap (M^\omega)^{\varphi^\omega} = \C$ or $Q' \cap (M^\omega)^{\varphi^\omega}$ is diffuse.
\end{newclaim}

\begin{proof}[Proof of the Claim]
We use the proof of \cite[Lemma 2.7]{Io12}. Put $\mathcal Q = Q' \cap (M^\omega)^{\varphi^\omega}$ and denote by $e \in \mathcal Z(\mathcal Q)$ the maximum central projection in $\mathcal Q$ such that $\mathcal Q e$ is discrete. We may represent $e = (e_n)^\omega$ by a sequence of projections $(e_n)_n \in \mathcal M^\omega(M)$. Put $\lambda = \varphi^\omega(e) = \lim_{n \to \omega} \varphi(e_n)$. Since $Q' \cap M = \C$, we have $e_n \to \lambda 1$ $\sigma$-weakly as $n \to \omega$.

Next, we construct by induction a sequence of projections $(f_m)_{m \geq 1}$ in $\mathcal Q$ such that 
\begin{equation}\label{equality-lambda}
\varphi^\omega(ef_i) = \lambda^2 \; \text{ and } \;  \varphi^\omega(ef_i f_j) = \lambda^3, \forall 1 \leq i < j.
\end{equation}
Indeed, assume that $f_1, \dots, f_m \in \mathcal Q$ have been constructed. For every $1 \leq j \leq m$, represent $f_j = (f_{j,n})^\omega$ by a sequence of projections $(f_{j, n})_n \in \mathcal M^\omega(M)$. Let $(x_i)_{i \in \N}$ be a $\|\cdot\|_\varphi^\sharp$-dense sequence in $\Ball(Q)$. Since $e = (e_n)^\omega \in (M^\omega)^{\varphi^\omega}$, since $\lim_{n \to \omega} \|e_n x_i - x_i e_n\|_\varphi^\sharp = 0$ for all $i \in \N$ and since $e_n \to \lambda 1$ $\sigma$-weakly as $n \to \omega$, we can find an increasing sequence $(k_n)_n$ in $\N$ such that for every $n \geq 1$,  we have
\begin{enumerate}
\item [(P1)] $\|e_{k_n} \varphi - \varphi e_{k_n}\| \leq \frac1n$, 
\item [(P2)] $\|e_{k_n} x_i - x_i e_{k_n}\|_\varphi^\sharp \leq \frac1n$ for all $1 \leq i \leq n$,
\item [(P3)] $|\varphi(e_n e_{k_n}) - \lambda \varphi(e_n)| \leq \frac1n$ and
\item [(P4)] $|\varphi(e_n f_{j, n} e_{k_n}) - \lambda \varphi(e_n f_{j, n})| \leq \frac1n$ for all $1 \leq j \leq m$.
\end{enumerate}

Property (P1) together with Proposition \ref{proposition-ultraproduct} imply that the sequence $(e_{k_n})_n$ lies in the multiplier algebra $\mathcal M^\omega(M)$ and $f = (e_{k_n})^\omega \in (M^\omega)^{\varphi^\omega}$. Property (P2) implies that $x_i f = f x_i$ for all $i \in \N$. Since $\{x_i : i \in \N\}$ is $\ast$-strongly dense in $\Ball(Q)$, we obtain that $f \in Q' \cap (M^\omega)^{\varphi^\omega} = \mathcal Q$. Finally, Property (P3) implies that $\varphi^\omega(ef) = \lambda \varphi^\omega(e) = \lambda^2$ and Property (P4) together with the induction hypothesis imply that $\varphi^\omega(ef_j f) = \lambda \varphi^\omega(e f_j) = \lambda^3$ for all $1 \leq j \leq m$. We can now put $f_{m + 1} = f$. This finishes the proof of the induction.

Define $p_m = f_m e$ which is a projection in $\mathcal Qe$. Observe that since $\mathcal Q e$ is a discrete tracial von Neumann algebra, $\mathcal Q e$ is $\ast$-isomorphic to a countable direct sum of finite dimensional factors and hence its unit ball $\Ball(\mathcal Q e)$ is $\|\cdot\|_{\varphi_e^\omega}$-compact, where $\varphi_e^\omega = \frac{\varphi^\omega(e \cdot e)}{\varphi^\omega(e)}$. Thus, we may choose a subsequence $(p_{m_k})_{k \geq 1}$ which is $\|\cdot\|_{\varphi_e^\omega}$-convergent in $\Ball(\mathcal Qe)$. By Cauchy-Schwarz inequality, for all $1 \leq j < k$, we have
$$|\varphi_e^\omega(p_{m_j} p_{m_k}) - \varphi_e^\omega(p_{m_j}) | = |\varphi_e^\omega(p_{m_j}(p_{m_k} - p_{m_j}))| \leq \|p_{m_j} - p_{m_k}\|_{\varphi_e^\omega}.$$
Taking the limit as $(j, k) \to \infty$ and using (\ref{equality-lambda}), we obtain $\lambda^2 = \lambda^3$. Therefore $\lambda \in \{ 0, 1 \}$ and so $e \in \{ 0, 1 \}$.

This implies that either $e = 0$ and $\mathcal Q$ is diffuse or $e = 1$ and $\mathcal Q$ is a discrete tracial von Neumann algebra. In the case when $\mathcal Q$ is a discrete tracial von Neumann algebra, we show that $\mathcal Q = \C$. Assume by contradiction that $\mathcal Q$ is a discrete tracial von Neumann algebra and that $\mathcal Q \neq \C$. 

Denote by $E_M : M^\omega \to M$ the canonical faithful normal conditional expectation. Recall  that $\varphi \circ E_M = \varphi^\omega$. Since $\mathcal Q \neq \C$, we may choose a projection $e \in \mathcal Q$ satisfying $\varphi^\omega(e) = \lambda$ with $\lambda \neq 0, 1$. We may represent $e = (e_n)^\omega  \in \mathcal Q$ by a sequence of projections $(e_n)_n \in \mathcal M^\omega(M)$. Observe that $E_M(e) = \lambda 1 = \sigma \text{-weak} \lim_{n \to \omega} e_n$. Then for all $y \in \Ball(M)$, we have
$$\|e - y\|_{\varphi^\omega} \geq \|e - E_M(e)\|_{\varphi^\omega} = \sqrt{\lambda - \lambda^2} > 0.$$

Put $\varepsilon = \frac{\sqrt{\lambda - \lambda^2}}{2}$. Put $e_1 = e \in \mathcal Q$. Next, we construct by induction a sequence of projections $e_m \in \mathcal Q$ such that $\|e_p - e_q\|_{\varphi^\omega} \geq \varepsilon$ for all $p, q \geq 1$ such that $p \neq q$. Assume that $e_1, \dots, e_{m} \in \mathcal Q$ have been constructed. For every $1 \leq j \leq m$, represent $e_j = (e_{j,n})^\omega$ by a sequence of projections $(e_{j, n})_n \in \mathcal M^\omega(M)$. Let $(x_i)_{i \in \N}$ be a $\|\cdot\|_\varphi^\sharp$-dense sequence in $\Ball(Q)$. Since $e = (e_n)^\omega \in (M^\omega)^{\varphi^\omega}$, since $\lim_{k \to \omega} \|e_k x_i - x_i e_k\|_\varphi^\sharp = 0$ for all $i \in \N$ and since $\lim_{k \to \omega} \|e_k - e_{j, n}\|_\varphi =  \|e - e_{j, n}\|_{\varphi^\omega} \geq 2 \varepsilon$ for all $1 \leq j \leq m$ and all $n \in \N$, we can find an increasing sequence $(k_n)_n$ in $\N$ such that for every $n \geq 1$,  we have
\begin{enumerate}
\item [(P1)] $\|e_{k_n} \varphi - \varphi e_{k_n}\| \leq \frac1n$, 
\item [(P2)] $\|e_{k_n} x_i - x_i e_{k_n}\|_\varphi^\sharp \leq \frac1n$ for all $1 \leq i \leq n$ and
\item [(P3)] $\|e_{k_n} - e_{j, n}\|_{\varphi} \geq \varepsilon$ for all $1 \leq j \leq m$.
\end{enumerate}

By the same reasoning as before, Properties (P1) and (P2) imply that $(e_{k_n})_n \in \mathcal M^\omega(M)$ and $f = (e_{k_n})^\omega \in \mathcal Q$. Moreover, Property (P3) implies that $\| f - e_j\|_{\varphi^\omega} \geq \varepsilon$ for all $1 \leq j \leq m$. We can now put $e_{m + 1} = f$. This finishes the proof of the induction.

So, we have constructed a sequence of projections $e_m \in \mathcal Q$ such that $\|e_p - e_q\|_{\varphi^\omega} \geq \varepsilon$ for all $p, q \geq 1$ such that $p \neq q$. This however contradicts the fact that $\Ball(\mathcal Q)$ is $\|\cdot\|_{\varphi^\omega}$-compact and finishes the proof of the Claim.
\end{proof}

Assume that $Q' \cap (M^\omega)^{\varphi^\omega} = \C$. Then by \cite[Lemma 5.4]{AH12}, we have that $Q' \cap M^\omega = \C$ or $Q' \cap M^\omega$ is a type ${\rm III_1}$ factor. Next, assume that $Q' \cap (M^\omega)^{\varphi^\omega}$ is diffuse. Then, using Proposition \ref{diffuse-centralizer}, we have that $Q' \cap M^\omega$ is diffuse. Therefore, either $Q' \cap M^\omega = \C$ or $Q' \cap M^\omega$ is diffuse. 
\end{proof}

\begin{prop}\label{gamma-amenable}
For every diffuse amenable von Neumann algebra $M$ with separable predual, the central sequence algebra $M' \cap M^\omega$ is diffuse.
\end{prop}

\begin{proof}
Let $M$ be any diffuse amenable von Neumann algebra with separable predual. There exists a sequence of pairwise orthogonal projections $z_n \in \mathcal Z(M)$ such that $\sum_{n} z_n = 1$, $M z_0$ is an amenable von Neumann algebra with a diffuse center and separable predual and $M z_n$ is a diffuse amenable factor with separable predual for every $n \geq 1$. It is obvious that $(Mz_0)' \cap (Mz_0)^\omega$ is diffuse. By the classification of amenable factors with separable predual (see \cite{Co72, Co74, Co75b, Co85, Ha84}), $Mz_n$ is hyperfinite and $(M z_n)' \cap (Mz_n)^\omega$ is diffuse for every $n \geq 1$. Therefore $M' \cap M^\omega = \bigoplus_{n} (M z_n)' \cap (M z_n)^\omega$ is diffuse. 
\end{proof}

\subsection*{An elementary fact on $\varepsilon$-orthogonality}

Let $\mathcal H$ be a complex Hilbert space and $\varepsilon \geq 0$. We say that two (not necessarily closed) subspaces $\mathcal K, \mathcal L \subset \mathcal H$ are $\varepsilon$-{\em orthogonal} and we denote by $\mathcal K \perp_\varepsilon \mathcal L$ if
$$|\langle \xi, \eta \rangle_{\mathcal H} | \leq \varepsilon \, \|\xi\|_{\mathcal H} \, \|\eta\|_{\mathcal H}, \; \forall \xi \in \mathcal K, \forall \eta \in \mathcal L.$$

Define the function
$$\delta : \left[0, \frac12 \right) \to \R_+ :  t \mapsto \frac{2 t}{\sqrt{1 - t - \sqrt{2} \, t \sqrt{1 - t}}}.$$ 
We will be using the following elementary fact regarding $\varepsilon$-orthogonality whose proof  can be found in \cite[Proposition 2.3]{Ho12a}.

\begin{prop}[\cite{Ho12a}]\label{projections}
Let $k \geq 1$. Let $0 \leq \varepsilon < 1$ such that $\delta^{\circ (k - 1)}(\varepsilon) < 1$. For all $1 \leq i \leq 2^k$, let $p_i \in \mathbf B(\mathcal H)$ be projections such that $p_i \mathcal H \perp_\varepsilon p_j \mathcal H$ for all $i, j \in \{1, \dots , 2^k\}$ such that $i \neq j$. Write $P_k = \bigvee_{i = 1}^{2^k} p_i$. Then for all $\xi \in \mathcal H$, we have
$$\sum_{i = 1}^{2^k} \|p_i \xi\|_{\mathcal H}^2 \leq \prod_{j = 0}^{k - 1} \left( 1 + \delta^{\circ j}(\varepsilon) \right)^2 \|P_k \xi\|_{\mathcal H}^2.$$
\end{prop}

\section{Asymptotic orthogonality in the ultraproduct framework}\label{asymptotic}

The key result of the paper is the following generalization of Popa's result \cite[Lemma 2.1]{Po83} regarding asymptotic orthogonality for free group factors to arbitrary free product von Neumann algebras. There are mainly two difficulties that arise in generalizing Popa's result \cite[Lemma 2.1]{Po83} to the setting of arbitrary free product von Neumann algebras. The first main difficulty is that the free product von Neumann algebra $(M, \varphi) = (M_1, \varphi_1) \ast (M_2, \varphi_2)$ is no longer assumed to be tracial. Hence, we need to work in the ultraproduct von Neumann algebra framework and carefully approximate elements in $M$ in the $\sigma$-strong topology by finite linear combinations of reduced words which are analytic with respect to the modular automorphism group $(\sigma_t^\varphi)$ (see also the proof of \cite[Proposition 3.5]{Ue11} where a similar method is used). The second main difficulty is that unlike the case of the free group factors, $M$ is no longer assumed to have a nice basis of unitary elements. To circumvent this issue, we will use $\varepsilon$-orthogonality techniques from \cite{Ho12a, Ho12b}.

\begin{thm}\label{key-thm}
Let $(M_1, \varphi_1)$ and $(M_2, \varphi_2)$ be $\sigma$-finite von Neumann algebras endowed with faithful normal states. Assume that the centralizer $M_1^{\varphi_1}$ is diffuse.  Denote by $(M, \varphi) = (M_1, \varphi_1) \ast (M_2, \varphi_2)$ the free product von Neumann algebra. 

Let $u \in \mathcal U(M_1^{\varphi_1})$ be any unitary such that $u^k \to 0$ weakly as $|k| \to \infty$. For every $x \in \{u\}' \cap M^\omega$ and every $y \in M \ominus M_1$, the elements $y(x - E_{M_1^\omega}(x))$, $(x - E_{M_1^\omega}(x))y$ and $y E_{M_1^\omega}(x) - E_{M_1^\omega}(x)y$ are pairwise $\varphi^\omega$-orthogonal in $M^\omega$.
\end{thm}

\begin{proof}
For every $i \in \{1, 2\}$, denote by $\mathcal A_i \subset M_i$ (resp.\ $\mathcal A \subset M$) the unital $\sigma$-strongly dense $\ast$-subalgebra of all the elements in $M_i$ (resp.\ $M$) which are analytic with respect to the modular automorphism group $(\sigma_t^{\varphi_i})$ (resp.\ $(\sigma_t^\varphi)$) (see Proposition \ref{modular-analytic}). Observe that for every $i \in \{1, 2\}$, $\mathcal A_i \subset \mathcal A$. Denote by $(\mathcal A_{i_1} \ominus \C) \cdots (\mathcal A_{i_n} \ominus \C)$ the set of all the reduced words of the form $a_1 \cdots a_n$ with $a_j \in \mathcal A_{i_j} \ominus \C$, $n \geq 1$ and $i_1 \neq \cdots \neq i_n$. The linear span of
$$\left\{1, (\mathcal A_{i_1} \ominus \C) \cdots (\mathcal A_{i_n} \ominus \C) : n \geq 1, i_1 \neq \cdots \neq i_n \right\}$$
forms a unital $\sigma$-strongly dense $\ast$-subalgebra of $M$. 

Using the existence of the normal conditional expectation $E_{M_1} : M \to M_1$, every $y \in M \ominus M_1$ can be approximated with respect to the $\sigma$-strong topology by a net $(y_\alpha)_{\alpha \in I}$ of finite linear combinations of reduced words in $(\mathcal A_{i_1} \ominus \C)  \cdots (\mathcal A_{i_n} \ominus \C)$ where $n \geq 1$, $2 \in \{i_1, \dots, i_n\}$ and $i_1 \neq \cdots \neq i_n$. Assume that for every $\alpha \in I$ and every $x \in \{u\}' \cap M^\omega$, $y_\alpha(x - E_{M_1^\omega}(x))$, $(x - E_{M_1^\omega}(x))y_\alpha$ and $y_\alpha E_{M_1^\omega}(x) - E_{M_1^\omega}(x)y_\alpha$ are pairwise $\varphi^\omega$-orthogonal in $M^\omega$. Then since $y_\alpha \to y$ $\sigma$-strongly as $\alpha \to \infty$, it follows that 
\begin{align*}
y_\alpha(x - E_{M_1^\omega}(x)) & \to y(x - E_{M_1^\omega}(x)) \\
(x - E_{M_1^\omega}(x))y_\alpha & \to (x - E_{M_1^\omega}(x))y \\
y_\alpha E_{M_1^\omega}(x) - E_{M_1^\omega}(x)y_\alpha & \to yE_{M_1^\omega}(x) - E_{M_1^\omega}(x)y
\end{align*}
$\sigma$-strongly as $\alpha \to \infty$. Therefore, $y(x - E_{M_1^\omega}(x))$, $(x - E_{M_1^\omega}(x))y$ and $yE_{M_1^\omega}(x) - E_{M_1^\omega}(x)y$ are pairwise $\varphi^\omega$-orthogonal in $M^\omega$. Using the previous discussion, we infer that it suffices to prove the result for 
$$y = \sum_{j = 1}^k w_j \; \text{ where } \; w_j = a_{j, 1} b_{j, 1} \cdots b_{j, n_j} a_{j, n_j + 1}$$ 
with $n_j \geq 1$, $a_{j, 1} = 1$ or $a_{j, 1} \in \mathcal A_1 \ominus \C$, $a_{j, n_j + 1} = 1$ or $a_{j, n_j + 1} \in \mathcal A_1 \ominus \C$, $a_{j, 2}, \dots, a_{j, n_j} \in \mathcal A_1 \ominus \C$ and $b_{j, 1}, \dots, b_{j, n_j} \in \mathcal A_2 \ominus \C$. We fix such an element $y \in M \ominus M_1$ until the end of the proof. Observe that for every $1 \leq j \leq k$, we have $w_j \in \mathcal A \ominus \C$ and 
$$\sigma_{-{\rm i}}^{\varphi}(w_j^*) = \sigma_{-{\rm i}}^{\varphi_1}(a_{j, n_j + 1}^*) \sigma_{-{\rm i}}^{\varphi_2}(b_{j, n_j}^*) \cdots \sigma_{-{\rm i}}^{\varphi_2}(b_{j, 1}^*) \sigma_{-{\rm i}}^{\varphi_1}(a_{j, 1}^*).$$ 
It follows that $\sigma_{-{\rm i}}^{\varphi}(w_j^*)$ is a reduced word containing at least one letter from $M_2 \ominus \C$.

Denote by $V \subset M_1$ the finite dimensional vector subspace generated by $1$ and by
\begin{itemize}
\item the first letters coming from $M_1 \ominus \C$ of the reduced words $w_i, w_i^*, \sigma_{-{\rm i}}^\varphi(w_i^*)$ and the first letters coming from $M_1 \ominus \C$ of all the reduced words arising in the finite linear decomposition of $w_j^*w_i$ into reduced words, for all $1 \leq i, j \leq k$, and
\item the last letters coming from $M_1 \ominus \C$ of the reduced words $w_i$ and the last letters coming from $M_1 \ominus \C$ of all the reduced words arising in the finite linear decomposition of $w_i \sigma_{-{\rm i}}^\varphi(w_j^*)$ into reduced words, for all $1 \leq i, j \leq k$.
\end{itemize}
Let $\ell = \dim(V)$ and choose elements $e_1, \dots, e_{\ell} \in V$ so that $(\Lambda_{\varphi_1}(e_i))_{i = 1}^\ell$ forms an orthonormal basis for $\Lambda_{\varphi_1}(V)$. By Gram-Schmidt process, choose a vector subspace $W \subset M_1$ so that 
$$\rL^2(M_1) = \Lambda_{\varphi_1}(V) \oplus \overline{\Lambda_{\varphi_1}(W)}.$$

We will be using the following notation:
\begin{itemize}
\item $\mathcal K_1 \subset \rL^2(M)$ is the closed subspace generated by the image under $\Lambda_\varphi$ of the linear span of all the reduced words in $(M_2 \ominus \C) \cdots (M_2 \ominus \C)$, $(V \ominus \C)(M_2 \ominus \C) \cdots (M_2 \ominus \C)$, $(M_2 \ominus \C) \cdots (M_2 \ominus \C)(M_1 \ominus \C)$ and $(V \ominus \C)(M_2 \ominus \C) \cdots (M_2 \ominus \C)(M_1 \ominus \C)$. Observe that 
$$\mathcal K_1 \cong \Lambda_\varphi(V) \otimes \rL^2((M_2 \ominus \C) \cdots (M_2 \ominus \C) M_1).$$
\item $\mathcal K_2 \subset \rL^2(M)$ is the closed subspace generated by the image under $\Lambda_\varphi$ of the linear span of all the reduced words in $W(M_2 \ominus \C) \cdots (M_2 \ominus \C)$ and $W (M_2 \ominus \C) \cdots (M_2 \ominus \C)(V \ominus \C)$. Observe that 
$$\mathcal K_2 \cong \rL^2(W(M_2 \ominus \C) \cdots (M_2 \ominus \C)) \otimes \Lambda_\varphi(V).$$
\item $\mathcal L \subset \rL^2(M)$ is the closed subspace generated by the image under $\Lambda_\varphi$ of the linear span of all the reduced words in $W (M_2 \ominus \C) \cdots (M_2 \ominus \C) W$. Observe that 
$$\rL^2(M_1) \oplus \mathcal K_1 \oplus \mathcal K_2 \oplus \mathcal L = \rL^2(M).$$
\end{itemize}

Let $u \in \mathcal U(M_1^{\varphi_1})$ such that $u^k \to 0$ weakly as $|k| \to \infty$ and put $T = u \, J_\varphi uJ_\varphi \in \mathcal U(\rL^2(M))$. Observe that since $u \in \mathcal U(M_1^{\varphi_1}) \subset \mathcal U(M^\varphi)$, we have $T \Lambda_\varphi(z) = \Lambda_\varphi(u z u^*)$ for all $z \in M$.

\begin{claim}\label{claim1}
For all $\varepsilon > 0$, there exists $k_0 \in \N$ such that for all $i \in \{1, 2\}$ and all $|k| \geq k_0$, we have $T^k \mathcal K_i \perp_\varepsilon \mathcal K_i$.
\end{claim}

\begin{proof}[Proof of Claim \ref{claim1}]
Let $\xi, \eta \in \mathcal K_1$ that we write $ \sum_{i = 1}^{\ell} \Lambda_\varphi(e_i) \otimes \xi_i$ and $\eta = \sum_{j = 1}^{\ell} \Lambda_\varphi(e_j) \otimes  \eta_j$ with $\xi_i, \eta_j \in \rL^2((M_2 \ominus \C) \cdots (M_2 \ominus \C) M_1)$. Observe that $\|\xi\|_\varphi^2 = \sum_{i = 1}^\ell \|\xi_i\|_\varphi^2 $ and $\|\eta\|_\varphi^2 = \sum_{j = 1}^\ell \|\eta_j\|_\varphi^2$. Since $u \in M^\varphi$, we have $T^k \xi = \sum_{i = 1}^{\ell} \Lambda_\varphi( u^k e_i) \otimes  J_\varphi u^k J_\varphi \xi_i$ and hence
$$|\langle T^k \xi, \eta \rangle_\varphi| \leq \sum_{i, j = 1}^{\ell} |\varphi(e_j^* u^k e_i)| \, \| \xi_i \|_\varphi \, \| \eta_j\|_\varphi.$$
Since $u^k \to 0$ weakly as $|k| \to \infty$, we may choose $k_1 \in \N$ such that for all $|k| \geq k_1$ and all $1 \leq i, j \leq \ell$, we have $|\varphi(e_j^* u^k e_i)| \leq \varepsilon/\ell$. By Cauchy-Schwarz inequality, for all $|k| \geq k_1$, we obtain $|\langle T^k \xi, \eta \rangle_\varphi| \leq \varepsilon  \| \xi \|_\varphi  \| \eta \|_\varphi$. 

Likewise let $\xi, \eta \in \mathcal K_2$ that we write $ \sum_{i = 1}^{\ell} \xi_i \otimes \Lambda_\varphi(e_i)$ and $\eta = \sum_{j = 1}^{\ell} \eta_j \otimes \Lambda_\varphi(e_j)$ with $\xi_i, \eta_j \in \rL^2(W (M_2 \ominus \C) \cdots (M_2 \ominus \C))$. Observe that $\|\xi\|_\varphi^2 = \sum_{i = 1}^\ell \|\xi_i\|_\varphi^2 $ and $\|\eta\|_\varphi^2 = \sum_{j = 1}^\ell \|\eta_j\|_\varphi^2$. Since $u \in M^\varphi$, we have $T^k \xi = \sum_{i = 1}^{\ell} u^k \xi_i \otimes \Lambda_\varphi( e_i u^{-k})$ and hence
$$|\langle T^k \xi, \eta \rangle_\varphi| \leq \sum_{i, j = 1}^{\ell} |\varphi(e_j^* e_i u^{-k})| \, \| \xi_i \|_\varphi \, \| \eta_j\|_\varphi.$$
Since $u^k \to 0$ weakly as $|k| \to \infty$, we may choose $k_2 \in \N$ such that for all $|k| \geq k_2$ and all $1 \leq i, j \leq \ell$, we have $|\varphi(e_j^* e_i u^{-k})| \leq \varepsilon/\ell$. By Cauchy-Schwarz inequality, for all $|k| \geq k_2$, we obtain $|\langle T^k \xi, \eta \rangle_\varphi| \leq \varepsilon \| \xi \|_\varphi \| \eta \|_\varphi$. 

Put $k_0 = \max (k_1, k_2)$. Then for all $i \in \{1, 2\}$ and all $|k| \geq k_0$, we have that $T^k \mathcal K_i \perp_\varepsilon \mathcal K_i$.
\end{proof}

\begin{claim}\label{claim2}
For all $i \in \{1, 2\}$ and all $(z_n)^\omega \in \{u\}' \cap M^\omega$, we have $$\lim_{n \to \omega} \|P_{\mathcal K_i}(\Lambda_\varphi(z_n))\|_\varphi =~0.$$ 
\end{claim}

\begin{proof}[Proof of Claim \ref{claim2}]
Let $i \in \{1, 2\}$ and $z = (z_n)^\omega \in \{u\}' \cap M^\omega$. We may assume that $\|z_n\|_\infty \leq 1$ for all $n \in \N$. For all $n \in \N$ and all $k \in \N$, we have 
$$
\begin{aligned}
\| P_{\mathcal K_i}(\Lambda_\varphi(z_n)) \|_\varphi^2 & = \|T^k P_{\mathcal K_i} (\Lambda_\varphi(z_n)) \|_\varphi^2 \\
& = \|T^k P_{\mathcal K_i} (\Lambda_\varphi(z_n)) - P_{T^k \mathcal K_i} (\Lambda_\varphi(z_n)) + P_{T^k \mathcal K_i} (\Lambda_\varphi(z_n))\|_\varphi^2 \\
& \leq 2 \|T^k P_{\mathcal K_i} (\Lambda_\varphi(z_n)) - P_{T^k \mathcal K_i} (\Lambda_\varphi(z_n))\|_\varphi^2 + 2 \| P_{T^k \mathcal K_i} (\Lambda_\varphi(z_n))\|_\varphi^2 \\
& = 2 \| P_{T^k \mathcal K_i} (\Lambda_\varphi(u^k z_n u^{-k} - z_n))\|_\varphi^2 + 2 \| P_{T^k \mathcal K_i} (\Lambda_\varphi(z_n))\|_\varphi^2 \\
& \leq 2 \|u^k z_n u^{-k} - z_n\|_\varphi^2 + 2 \| P_{T^k \mathcal K_i} (\Lambda_\varphi(z_n))\|_\varphi^2.
\end{aligned}
$$
Fix $K \geq 1$. Choose $\varepsilon > 0$ very small according to \cite[Proposition 2.3]{Ho12a} so that $\prod_{j = 0}^{K - 1}(1 + \delta^{\circ j}(\varepsilon))^2 \leq 2$. Then choose a subset $\mathcal G \subset \N$ of $2^K$ integers such that two distinct integers in $\mathcal G$ are at least at distance $k_0$ from one another. By Claim \ref{claim1}, we obtain $T^{k_1} \mathcal K_i \perp_\varepsilon T^{k_2} \mathcal K_i$ for all $k_1, k_2 \in \mathcal G$ such that $k_1 \neq k_2$. Thus, we obtain
$$2^K \| P_{\mathcal K_i}(\Lambda_\varphi(z_n)) \|_\varphi^2 \leq 2 \sum_{k \in \mathcal G} \|u^k z_n u^{-k} - z_n\|_\varphi^2 + 4 \| z_n \|_\varphi^2.$$
Since $\mathcal G$ is finite, we have $\lim_{n \to \omega} \| P_{\mathcal K_i}(\Lambda_\varphi(z_n)) \|_\varphi^2 \leq 2^{2 - K}$ for all $K \geq 1$. Therefore, we obtain $\lim_{n \to \omega} \| P_{\mathcal K_i}(\Lambda_\varphi(z_n)) \|_\varphi = 0$.
\end{proof}

\begin{claim}\label{claim3}
The subspaces $y \, \mathcal L$, $J_\varphi \sigma_{-{\rm i}/2}^\varphi(y^*) J_\varphi \,  \mathcal L$ and $y \,  \rL^2(M_1) + J_\varphi \sigma_{-{\rm i}/2}^\varphi(y^*) J_\varphi \, \rL^2(M_1)$ are pairwise orthogonal in $\rL^2(M)$.
\end{claim}

\begin{proof}[Proof of Claim \ref{claim3}]
Recall that $y = \sum_{j = 1}^k w_j$ where $w_j = a_{j, 1} b_{j, 1} \cdots b_{j, n_j} a_{j, n_j + 1}$ 
with $n_j \geq 1$, $a_{j, 1} = 1$ or $a_{j, 1} \in \mathcal A_1 \ominus \C$, $a_{j, n_j + 1} = 1$ or $a_{j, n_j + 1} \in \mathcal A_1 \ominus \C$, $a_{j, 2}, \dots, a_{j, n_j} \in \mathcal A_1 \ominus \C$ and $b_{j, 1}, \dots, b_{j, n_j} \in \mathcal A_2 \ominus \C$. Observe that
\begin{align}
y \, \mathcal L & \subset \overline{\spn  \left\{ \Lambda_\varphi(w_j W (M_2 \ominus \C) \cdots (M_2 \ominus \C) W) : 1 \leq j \leq k \right\}} \label{line1} \\
J_\varphi \sigma_{-{\rm i}/2}^\varphi(y^*) J_\varphi \, \mathcal L & \subset  \overline{\spn \left\{ \Lambda_\varphi(W (M_2 \ominus \C) \cdots (M_2 \ominus \C) Ww_j) : 1 \leq j \leq k \right\}} \label{line2} 
\end{align}
and
\begin{equation} \label{line3}
y \,  \rL^2(M_1) + J_\varphi \sigma_{-{\rm i}/2}^\varphi(y^*) J_\varphi \, \rL^2(M_1)  \subset \overline{\spn \left\{ \Lambda_\varphi(w_i M_1), \Lambda_\varphi(M_1 w_j) : 1 \leq i, j \leq k \right \}}.
\end{equation}

Let $1 \leq i \leq k$. Observe that by the choice of the vector subspace $W \subset M_1$, any letter $v \in W$ is $\varphi$-orthogonal in $M$ to the first letter of the reduced word $w_i^*$ and to the first letter of the reduced word $\sigma_{-{\rm i}}^\varphi(w_i^*)$. Hence $w_i v$ is a reduced word starting with the first letter of $w_i$ and ending with a letter from $M_1 \ominus \C$ and $vw_i$ is a reduced word starting with a letter from $M_1 \ominus \C$ and ending with the last letter of $w_i$. Moreover both $vw_i$ and $w_iv$ contain at least one letter from $M_2 \ominus \C$.

Let $1 \leq i, j \leq k$. By the choice of the vector subspace $W \subset M_1$ and the remark above, the first letter of any reduced word $w_iv$ with $v \in W$ is $\varphi$-orthogonal to $W$ in $M$. This implies that $W (M_2 \ominus \C) \cdots (M_2 \ominus \C) Ww_j$ and $w_i W (M_2 \ominus \C) \cdots (M_2 \ominus \C) W$ are $\varphi$-orthogonal in $M$. Since this holds for all $1 \leq i, j \leq k$, using (\ref{line1}) and (\ref{line2}), we obtain that the subspaces $y \, \mathcal L$ and $J_\varphi \sigma_{-{\rm i}/2}^\varphi(y^*) J_\varphi \,  \mathcal L$ are orthogonal in $\rL^2(M)$.

Let $1 \leq i, j \leq k$. If $n_i \leq n_j$, then any element in $w_i M_1$ is a finite linear combination of reduced words which have at most $n_i$ letters from $M_2 \ominus \C$ while a reduced word in $w_j W(M_2 \ominus \C) \cdots (M_2 \ominus \C) W$ has at least $n_j + 1$ letters from $M_2 \ominus \C$. This implies that  $w_j W (M_2 \ominus \C) \cdots (M_2 \ominus \C) W$ and $w_i M_1$ are $\varphi$-orthogonal in $M$.  If $n_i > n_j$, then $w_j^* w_i$ is a finite linear combination of reduced words whose first letter is $\varphi$-orthogonal to $W$ in $M$ and which contain at least one letter from $M_2 \ominus \C$. It follows that any element in $w_j^* w_i M_1$ is a finite linear combination of reduced words whose first letter is $\varphi$-orthogonal to $W$ in $M$. Again, this implies that $w_j W (M_2 \ominus \C) \cdots (M_2 \ominus \C) W$ and $w_i M_1$ are $\varphi$-orthogonal in $M$. Next, since $w_i$ contains at least one letter from $M_2 \ominus \C$ and by the choice of the vector subspace $W \subset M_1$, any element in $M_1 w_i$ is a finite linear combination of reduced words whose last letter is $\varphi$-orthogonal to $W$ in $M$. This implies that $w_j W (M_2 \ominus \C) \cdots (M_2 \ominus \C) W$ and $M_1w_i$ are $\varphi$-orthogonal in $M$. Since the previous reasoning holds for all $1 \leq i, j \leq k$, using (\ref{line1}) and (\ref{line3}), we obtain that the subspaces $y \, \mathcal L$ and $y \,  \rL^2(M_1) + J_\varphi \sigma_{-{\rm i}/2}^\varphi(y^*) J_\varphi \, \rL^2(M_1)$ are orthogonal in $\rL^2(M)$.

Let $1 \leq i, j \leq k$. Since $w_i$ contains at least one letter from $M_2 \ominus \C$ and by the choice of the vector subspace $W \subset M_1$, any element in $w_i M_1$ is a finite linear combination of reduced words whose first letter is $\varphi$-orthogonal to $W$ in $M$. This implies that $W (M_2 \ominus \C) \cdots (M_2 \ominus \C) Ww_j$ and $w_i M_1$ are $\varphi$-orthogonal in $M$. Next, if $n_i \leq n_j$, then any element in $M_1 w_i$ is a finite linear combination of reduced words which have at most $n_i$ letters from $M_2 \ominus \C$ while a reduced word in $W(M_2 \ominus \C) \cdots (M_2 \ominus \C) Ww_j$ has at least $n_j + 1$ letters from $M_2 \ominus \C$. This implies that $W (M_2 \ominus \C) \cdots (M_2 \ominus \C) W w_j$ and $M_1 w_i$ are $\varphi$-orthogonal in $M$. If $n_i > n_j$, then $w_i \sigma_{- {\rm i}}^\varphi(w_j^*)$ is a finite linear combination of reduced words whose last letter is $\varphi$-orthogonal to $W$ in $M$ and which contain at least one letter from $M_2 \ominus \C$. It follows that any element in $M_1 w_i \sigma_{- {\rm i}}^\varphi(w_j^*)$ is a finite linear combination of reduced words whose last letter is $\varphi$-orthogonal to $W$ in $M$. Using Proposition \ref{modular-analytic}, this implies again that $W (M_2 \ominus \C) \cdots (M_2 \ominus \C) W w_j$ and $M_1w_i$ are $\varphi$-orthogonal in $M$. Since the previous reasoning holds for all $1 \leq i, j \leq k$, using (\ref{line2}) and (\ref{line3}), we obtain that the subspaces $J_\varphi \sigma_{-{\rm i}/2}^\varphi(y^*) J_\varphi \, \mathcal L$ and $y \,  \rL^2(M_1) + J_\varphi \sigma_{-{\rm i}/2}^\varphi(y^*) J_\varphi \, \rL^2(M_1)$ are orthogonal in $\rL^2(M)$. This finishes the proof of Claim~\ref{claim3}.
\end{proof}

We are now ready to finish the proof of Theorem \ref{key-thm}. Let $x \in \{u\}' \cap M^\omega$ and put $z = x - E_{M_1^\omega}(x)$. Observe that since $u \in M_1 \subset M_1^\omega$, we have $z \in \{u\}' \cap (M^\omega \ominus M_1^\omega)$. Write $z = (z_n)^\omega \in \{u\}' \cap (M^\omega \ominus M_1^\omega)$ with $z_n = x_n - E_{M_1}(x_n)$. By Claim \ref{claim2} and since $y$ is analytic with respect to the modular automorphism group $(\sigma_t^\varphi)$, we obtain
\begin{align*}
\Lambda_{\varphi^\omega}(y z) &= (\Lambda_{\varphi}(y z_n))_\omega = (y \, \Lambda_\varphi(z_n))_\omega \\
&= (y \, P_{\mathcal L} (\Lambda_\varphi(z_n)))_\omega \in \rL^2(M)^\omega \\
\Lambda_{\varphi^\omega}(z y) & = (\Lambda_\varphi(z_n y))_\omega = (J_\varphi \sigma_{-{\rm i}/2}^\varphi(y)^*J_\varphi \, \Lambda_\varphi(z_n))_\omega \\
&= (J_\varphi \sigma_{-{\rm i}/2}^\varphi(y)^*J_\varphi \, P_{\mathcal L}(\Lambda_\varphi(z_n)))_\omega \in \rL^2(M)^\omega \\ 
\Lambda_{\varphi^\omega}(y E_{M_1^\omega}(x) - E_{M_1^\omega}(x)y) &= (\Lambda_\varphi(yE_{M_1}(x_n) - E_{M_1}(x_n)y))_\omega \\
&= ((y - J_\varphi \sigma_{-{\rm i}/2}^\varphi(y)^*J_\varphi) \, \Lambda_{\varphi}(E_{M_1}(x_n) ))_\omega \in \rL^2(M)^\omega.
\end{align*}

Using Claim \ref{claim3} for every $n \in \N$ and using the ultraproduct Hilbert space structure of $\rL^2(M)^\omega$, we obtain that $\Lambda_{\varphi^\omega}(y(x - E_{M_1^\omega}(x)))$, $\Lambda_{\varphi^\omega}((x - E_{M_1^\omega}(x))y)$ and $\Lambda_{\varphi^\omega}(y E_{M_1^\omega}(x) - E_{M_1^\omega}(x)y)$ are pairwise orthogonal in $\rL^2(M)^\omega$. This implies that $y(x - E_{M_1^\omega}(x))$, $(x - E_{M_1^\omega}(x))y$ and $y E_{M_1^\omega}(x) - E_{M_1^\omega}(x)y$ are pairwise $\varphi^\omega$-orthogonal in $M^\omega$.
\end{proof}

\section{Proofs of the main results}\label{proofs}

\subsection{Proof of Theorem \ref{thmA} and Corollaries \ref{corB} and \ref{corC}} 

\begin{proof}[Proof of Theorem \ref{thmA}]
Let $M_1 \subset Q \subset M$ be any intermediate von Neumann subalgebra such that $Q' \cap M^\omega$ is diffuse. Since $M_1^{\varphi_1}$ is diffuse, by \cite[Corollary 3.2]{Ue11}, we have $Q' \cap M \subset M_1' \cap M \subset M_1$ and so $Q' \cap M = \mathcal Z(Q) =  Q' \cap M_1 \subset \mathcal Z(M_1)$.

First, denote by $z \in Q' \cap M$ the maximum projection such that $M_1z = Qz$. We show that $z = 1$. Assume by contradiction that $z \neq 1$ and put $q = z^\perp = 1 - z \in Q' \cap M = \mathcal Z(Q)$. We have $q \neq 0$ and $Qq \ominus M_1q \neq 0$. Denote by $\mathcal J$ the nonzero $\sigma$-strongly closed two-sided ideal in $Qq$ generated by $Qq \ominus M_1q$. Let $e \in \mathcal Z(Qq) = \mathcal Z(Q) q$ be the unique nonzero central projection in $Qq$ such that $\mathcal J = Qe$. We necessarily have $e = q$. Indeed otherwise we have $q - e \neq 0$ and by the choice of the projection $z \in Q' \cap M$, we have $Q(q - e) \ominus M_1(q - e) \neq 0$. Now let $y \in Q(q - e) \ominus M_1(q - e)$ such that $y \neq 0$. Since $y \in Q q \ominus M_1 q$, we obtain $y \in \mathcal J$ and so $y = ye$. However since $y \in Q(q - e) \ominus M_1(q - e)$, we also obtain $y = y(q - e)$ and thus $y = 0$. This is a contradiction. Thus, we have $e = q$.

Next, we show that $(Qq)' \cap (qMq)^\omega \subset (M_1q)^\omega$. Indeed, let $x \in (Qq)' \cap (qMq)^\omega \subset M_1' \cap M^\omega$. For every $y \in Q q \ominus M_1 q \subset M \ominus M_1$, we have
$$0 = yx - xy = y(x - E_{M_1^\omega}(x)) - (x - E_{M_1^\omega}(x)) y + (y E_{M_1^\omega}(x) - E_{M_1^\omega}(x) y).$$
By Theorem \ref{key-thm}, $y(x - E_{M_1^\omega}(x))$, $(x - E_{M_1^\omega}(x)) y$ and $(y E_{M_1^\omega}(x) - E_{M_1^\omega}(x) y)$ are pairwise $\varphi^\omega$-orthogonal in $M^\omega$. By Pythagora's theorem, we obtain $y(x - E_{M_1^\omega}(x)) = 0$. Likewise, for every $a \in Qq$ and every $y \in Qq \ominus M_1q$, we have $a \, y \,  (x - E_{M_1^\omega}(x)) = 0$ and since $y E_{M_1}(a) \in Qq \ominus M_1q$ and $a - E_{M_1}(a) \in Qq \ominus M_1q$, we also have
$$
y \, a \,  (x - E_{M_1^\omega}(x)) = y \, E_{M_1}(a) \, (x - E_{M_1^\omega}(x)) + y \, (a - E_{M_1}(a)) \, (x - E_{M_1^\omega}(x)) = 0.
$$
This implies that for every $y \in \mathcal J$, we have $y (x - E_{M_1^\omega}(x)) = 0$ hence $q (x - E_{M_1^\omega}(x)) = 0$. Therefore, $x = E_{(M_1q)^\omega}(x) \in (M_1q)^\omega$.

Now we have that $(Qq)' \cap (qMq)^\omega = (Qq)' \cap (M_1q)^\omega$. Since $Q' \cap M^\omega$ is diffuse and since $(Qq)' \cap (qMq)^\omega = q(Q' \cap M^\omega)q$, we have that $(Qq)' \cap (M_1q)^\omega$ is diffuse as well. This implies that there exists a net of unitaries $U_\alpha \in \mathcal U((Qq)' \cap(M_1q)^\omega)$ such that $U_\alpha \to 0$ weakly as $\alpha \to \infty$. We may represent every $U_\alpha \in \mathcal U((Qq)' \cap (M_1 q)^\omega)$ by a sequence of elements $(u_n^\alpha)_n \in \mathcal M^\omega(M_1q)$ such that $u_n^\alpha \in \Ball(M_1q)$ for every $\alpha$ and every $n \in \N$, $U_\alpha = (u^\alpha_n)^\omega$ for every $\alpha$ and $y u_n^\alpha - u_n^\alpha y \to 0$ $\ast$-strongly as $n \to \omega$ for every $\alpha$ and every $y \in Qq$. 

Define the directed set 
$$\mathcal I = \{i = (\varepsilon, \mathcal F, \mathcal G ) : \varepsilon > 0, \, \mathcal F \subset M_1q \text{ and } \mathcal G \subset Qq \text{ are finite subsets}\}$$
with order relation given by 
$$(\varepsilon_1, \mathcal F_1, \mathcal G_1) \leq (\varepsilon_2, \mathcal F_2, \mathcal G_2) \; \text{ if and only if } \; \varepsilon_2 \leq \varepsilon_1, \, \mathcal F_1 \subset \mathcal F_2 \text{ and } \mathcal G_1 \subset \mathcal G_2.$$
Let $i = (\varepsilon, \mathcal F, \mathcal G) \in \mathcal I$. Since $U_\alpha \to 0$ weakly as $\alpha \to \infty$, there exists $\alpha$ such that $|\varphi^\omega(b^* U_\alpha a)| \leq \varepsilon/2$ for all $a, b \in \mathcal F$. Since $U_\alpha = (u^\alpha_n)^\omega \in \mathcal U((Qq)' \cap (M_1q)^\omega)$, for all $a, b \in \mathcal F$ and all $c \in \mathcal G$, we have 
\begin{align*}
\frac{\varepsilon}{2} \geq |\varphi^\omega(b^* U_\alpha a)| &= \lim_{n \to \omega} |\varphi(b^* u_n^\alpha a)| \\
\|a\|_\varphi = \|U_\alpha a\|_{\varphi^\omega} &= \lim_{n \to \omega} \|u_n^\alpha a\|_\varphi \\
0 = \|c U_\alpha - U_\alpha c\|_{\varphi^\omega} &= \lim_{n \to \omega} \|c u^\alpha_n - u^\alpha_n c\|_\varphi.
\end{align*}
Since $\mathcal F \subset M_1q$ and $\mathcal G \subset Qq$ are finite subsets, there exists $n = n(\alpha)$ such that
$$\max \left\{|\|a\|_\varphi - \|u^\alpha_{n(\alpha)} a\|_\varphi|, \|cu^{\alpha}_{n(\alpha)} - u^\alpha_{n(\alpha)} c\|_\varphi, |\varphi(b^* u^{\alpha}_{n(\alpha)} a)| : a, b \in \mathcal F, c \in \mathcal G \right\} \leq \varepsilon.$$
Put $w_i = u^{\alpha}_{n(\alpha)} \in \Ball (M_1 q)$. Thus, $(w_i)_{i \in \mathcal I}$ is a net of elements in $\Ball(M_1q)$ such that 
\begin{itemize}
\item [(P1)] $\lim_{i \in \mathcal I} \|w_i a\|_\varphi = \|a\|_\varphi$ for all $a \in M_1q$.
\item [(P2)] $\lim_{i \in \mathcal I}  \|c w_i - w_i c\|_\varphi = 0$ for all $c \in Qq$.
\item [(P3)] $\lim_{i \in \mathcal I} |\varphi(b^* w_i a)| = 0$ for all $a, b \in M_1q$.
\end{itemize}

Put $\mathcal E = \spn(\{q (M_{i_1} \ominus \C) \cdots (M_{i_n} \ominus \C) q : n \geq 1, 2 \in \{i_1, \dots, i_n\} \text{ and } i_1 \neq \cdots \neq i_n\})$. Observe that $\mathcal E$ is $\sigma$-strongly dense in $qMq \ominus M_1 q$.
\begin{newclaim}
The following hold true.
\begin{enumerate}
\item For all $a, b \in \mathcal E$, we have
$$\lim_{i \in \mathcal I} \|E_{M_1q}(b^* w_i a)\|_\varphi = 0.$$ 
\item For all $b \in \mathcal E$ and all $y \in qMq \ominus M_1q$, we have
$$\lim_{i \in \mathcal I} \|E_{M_1q}(b^* w_i y)\|_\varphi = 0.$$ 
\end{enumerate}
\end{newclaim}

\begin{proof}[Proof of the Claim]
$(1)$ By linearity, it suffices to prove the result for all the elements $a, b \in \mathcal E$ of the form $a = a_1 \cdots a_{2m + 1}$ and $b = b_1 \cdots b_{2n + 1}$ with $m, n \geq 1$, $a_1= q$ or $a_1 \in M_1q \ominus \C q$, $a_{2m + 1} = q$ or $a_{2m + 1} \in M_1q \ominus \C q$, $b_1 = q$ or $b_1 \in M_1q \ominus \C q$, $b_{2n + 1} = q$ or $b_{2 n+ 1} \in M_1q \ominus \C q$, $a_2, \dots, a_{2m}, b_2, \dots, b_{2n} \in M_2 \ominus \C$ and $a_3, \dots a_{2m - 1}, b_3, \dots, b_{2n - 1} \in M_1 \ominus \C$. We have
$$b^* w_i a = b_{2n + 1}^* \cdots b_2^* \, (b_1^*w_i a_1) \, a_2 \cdots a_{2m + 1}.$$
By the freeness property, we have
$$E_{M_1}(b^* w_i a) = \varphi(b_1^* w_i a_1) \, E_{M_1}(b_{2n + 1}^* \cdots b_2^* \, a_2 \cdots a_{2m + 1}).$$
Using property (P3) of the net $(w_i)_{i \in \mathcal I}$, we obtain $\lim_{i \in \mathcal I} \|E_{M_1q}(b^* w_i a)\|_\varphi = 0$.

$(2)$ Let $y \in qMq \ominus M_1q$ and $b \in \mathcal E$. We may assume that $\|b\|_\infty \leq 1$. Since $\mathcal E$ is $\sigma$-strongly dense in $qMq \ominus M_1q$, for every $\varepsilon > 0$, there exists $a \in \mathcal E$ such that $\|y - a\|_\varphi \leq \varepsilon/2$. Thus, for every $i \in \mathcal I$, we have
$$\|E_{M_1}(b^*w_i (y - a))\|_\varphi \leq \|b^* w_i (y - a)\|_\varphi \leq \|y - a\|_\varphi \leq \varepsilon.$$
Using the first part of the proof, this implies that $\limsup_{i \in \mathcal I} \|E_{M_1q}(b^* w_i y)\|_\varphi \leq \varepsilon$. Since $\varepsilon > 0$ is arbitrary, we obtain $\lim_{i \in \mathcal I} \|E_{M_1q}(b^* w_i y)\|_\varphi = 0$. This finishes the proof of the Claim.
\end{proof}

Let $b \in \mathcal E$ and $y \in Qq \ominus M_1q$. Using the properties (P1) and (P2) of the net $(w_i)_{i \in \mathcal I}$, we obtain
\begin{align*}
\|E_{M_1q}(b^* y)\|_\varphi &= \lim_{i \in \mathcal I} \|w_i \, E_{M_1q}(b^* y)\|_\varphi \; \text{ using (P1) for } a = E_{M_1q}(b^*y) \\
&= \lim_{i \in \mathcal I} \|E_{M_1q}(b^* y) \, w_i\|_\varphi  \; \text{ using (P2) for } c = E_{M_1q}(b^*y) \\
&= \lim_{i \in \mathcal I} \|E_{M_1q}(b^* y \, w_i) \|_\varphi \; \text{ since } w_i \in M_1q \\
&= \lim_{i \in \mathcal I} \|E_{M_1q}(b^* w_i \, y) \|_\varphi \; \text{ using (P2) for } c = y \\
& = 0 \quad \quad \quad \quad \quad \quad \quad \quad \; \, \text{ using item (2) of the Claim}. 
\end{align*}
Since $\mathcal E$ is $\sigma$-strongly dense in $qMq \ominus M_1q$, we may choose a net $(b_j)_{j \in J}$ in $\mathcal E$ such that $b_j^* \to y^*$ $\sigma$-strongly as $j \to \infty$. Since $E_{M_1q} : qMq \to M_1q$ is $\sigma$-strongly continuous, we obtain that $E_{M_1q}(b_j^*y) \to E_{M_1q}(y^*y)$ $\sigma$-strongly as $j \to \infty$ and hence $E_{M_1q}(y^* y) = 0$. This implies that $y^*y = 0$ and hence $y = 0$. Since $y \in Qq \ominus M_1q$ is arbitrary, we derive that $M_1q = Qq$. This contradicts the maximality of the projection $z \in Q' \cap M$ and finishes the proof of Theorem~\ref{thmA}.
\end{proof}

\begin{proof}[Proof of Corollary \ref{corB}]
Let $M_1 \subset Q \subset M$ be any intermediate von Neumann subalgebra with faithful normal conditional expectation $E_Q : M \to Q$. Denote by $E_{M_1} : M \to M_1$ the unique $\varphi$-preserving normal conditional expectation. Since $M_1^{\varphi_1}$ is diffuse, we have $M_1' \cap M \subset M_1$ by \cite[Corollary 3.2]{Ue11} and hence $E_{M_1}$ is the unique faithful normal conditional expectation from $M$ to $M_1$ by \cite[Th\'eor\`eme 1.5.5]{Co72}. Since $E_{M_1} \circ E_Q$ is a faithful normal conditional expectation from $M$ to $M_1$, we have $E_{M_1} \circ E_Q = E_{M_1}$. This implies that for every $x \in M$, we have
$$\varphi(E_Q(x)) = \varphi(E_{M_1}(E_Q(x))) = \varphi((E_{M_1} \circ E_Q)(x)) = \varphi(E_{M_1}(x)) = \varphi(x).$$
By \cite[Theorem IX.4.2]{Ta03}, we obtain that $Q$ is globally invariant under the modular automorphism group $(\sigma_t^\varphi)$.

Since $Q' \cap M = \mathcal Z(Q)$ is abelian, there exists a sequence of pairwise orthogonal projections $q_n \in Q' \cap M \subset \mathcal Z(M_1)$ such that $\sum_{n} q_n = 1$, $(Qq_0)' \cap q_0 M q_0 = (Q' \cap M)q_0$ is a diffuse abelian von Neumann algebra and $Qq_n$ is a diffuse factor such that $(Qq_n)' \cap q_n M q_n = (Q' \cap M)q_n = \C q_n$ for every $n \geq 1$. Define
$$\mathcal I = \{0\} \cup \left\{n \geq 1 : (Qq_n)' \cap (q_n M q_n)^\omega \text{ is diffuse} \right\}.$$

Put $z_0 = \sum_{n \in \mathcal I} q_n$ and $N = (\C z_0 \oplus M_1 z_0^\perp) \vee M_2$. If $z_0 = 0$, then $M_1 z_0 = Q z_0$. Otherwise, by \cite[Lemma 2.2]{Ue11}, we have that $M_1z_0$ and $z_0Nz_0$ generate $z_0 M z_0$ and are free in $z_0 M z_0$ with respect to the state $\varphi_{z_0} = \frac{\varphi(z_0 \cdot z_0)}{\varphi(z_0)}$. Thus, we have
$$(z_0 M z_0, \varphi_{z_0}) = (M_1 z_0, \varphi_{z_0}) \ast (z_0 N z_0, \varphi_{z_0}).$$
Moreover, the intermediate subalgebra $M_1z_0 \subset Qz_0 \subset z_0 M z_0$ is globally invariant under the modular automorphism group $(\sigma_t^{\varphi_{z_0}})$ and we have
\begin{equation}\label{inclusion-invariant}
\bigoplus_{n \in \mathcal I} \; (Q q_n)' \cap (q_n M q_n)^\omega \subset (Q z_0)' \cap (z_0 M z_0)^\omega.
\end{equation}
Since $Q \subset M$ is globally invariant under the modular automorphism group $(\sigma^{\varphi}_t)$ and since $q_n \in M^{\varphi}$ for all $n \in \N$, we have that both $\bigoplus_{n \in \mathcal I} \; (Q q_n)' \cap (q_n M q_n)^\omega$ and $(Q z_0)' \cap (z_0 M z_0)^\omega$ are globally invariant under the modular automorphism group $(\sigma^{\varphi_{z_0}^\omega}_t)$. Therefore, the inclusion (\ref{inclusion-invariant}) is with expectation. Since $\bigoplus_{n \in \mathcal I} \; (Q q_n)' \cap (q_n M q_n)^\omega$ is diffuse, so is $(Q z_0)' \cap (z_0 M z_0)^\omega$ by Proposition \ref{diffuse-centralizer}. Applying Theorem \ref{thmA} to the intermediate von Neumann subalgebra $M_1z_0 \subset Q z_0 \subset z_0Mz_0$, we obtain $M_1z_0 = Qz_0$.

For every $n \notin \mathcal I$, $(Qq_n)' \cap (q_n M q_n)^\omega$ is not diffuse. By Proposition \ref{gamma}, we obtain that $(Qq_n)' \cap (q_n M q_n)^\omega = \C q_n$. In particular, since $Q \subset M$ is with expectation, we have $(Qq_n)' \cap  (Qq_n)^\omega \subset (Qq_n)' \cap  (q_n M q_n)^\omega$. Thus, we have $(Qq_n)' \cap  (Qq_n)^\omega = \C q_n$ and so $Qq_n$ is a full nonamenable factor by Proposition \ref{gamma-amenable}. This finishes the proof of Corollary \ref{corB}.
\end{proof}

\begin{proof}[Proof of Corollary \ref{corC}]
Let $M_1$ be any diffuse amenable von Neumann algebra with separable predual. Choose a faithful normal state $\varphi_1$ on $M_1$ such that the centralizer $M_1^{\varphi_1}$ is diffuse (see Proposition \ref{diffuse-centralizer}). Define $M_2 = R_\infty$ to be the unique hyperfinite type ${\rm III_1}$ factor endowed with any faithful normal state $\varphi_2$. Then by \cite[Theorem 3.4]{Ue11}, the free product $(M, \varphi) = (M_1, \varphi_1) \ast (M_2, \varphi_2)$ is a full nonamenable type ${\rm III_1}$ factor. Moreover $M_1 \subset M$ is with expectation.

Let $M_1 \subset Q \subset M$ be any intermediate amenable von Neumann algebra with expectation. By Corollary \ref{corB}, we obtain that $M_1 = Q$.
\end{proof}

\subsection{Proof of Theorem \ref{thmD}}
We recall Popa's intertwining-by-bimodules theory that will play a crucial role in the proof of Theorem \ref{thmD}. Let  $M$ be a tracial von Neumann algebra together with $A \subset 1_A M 1_A$ and $B \subset 1_B M 1_B$ von Neumann subalgebras. Following \cite{Po01, Po03}, we say that $A$ {\em embeds into} $B$ {\em inside} $M$ and denote by $A \preceq_M B$ if one of the following equivalent conditions is satisfied:
\begin{itemize}
\item There exist projections $p \in A$ and $q \in B$, a nonzero partial isometry $v \in pMq$ and a unital normal $\ast$-homomorphism $\varphi : pAp \to qBq$ such that $a v = v \varphi(a)$ for all $a \in p A p$.

\item There exist $\ell \geq 1$, a projection $q \in \mathbf M_\ell (B)$, a nonzero partial isometry $v \in \mathbf M_{1, \ell}(1_A M)q$ and a unital normal $\ast$-homomorphism $\varphi : A \to q \mathbf M_\ell (B) q$ such that $a v = v \varphi(a)$ for all $a \in A$.

\item There is no net of unitaries $(w_i)_{i \in I}$ in $\mathcal U(A)$ such that $E_B(x^* w_i y) \to 0$ $\ast$-strongly as $i \to \infty$ for all $x, y \in p M q$.
\end{itemize}

We first prove the following intermediate result which can be regarded as a generalization of Theorem \ref{thmA} in the case of tracial free product von Neumann algebras.

\begin{thm}\label{step1}
Let $(M_1, \tau_1)$ and $(M_2, \tau_2)$ be von Neumann algebras with separable predual endowed with faithful normal tracial states. Assume that $M_1$ is diffuse. Denote by $(M, \tau) = (M_1, \tau_1) \ast (M_2, \tau_2)$ the tracial free product von Neumann algebra. 

For every von Neumann subalgebra $Q \subset M$ such that $Q \cap M_1$ and $Q' \cap M^\omega$ are diffuse, we have $Q \subset M_1$.
\end{thm}

\begin{proof}
Let $Q \subset M$ be any von Neumann subalgebra such that $Q \cap M_1$ and $Q' \cap M^\omega$ are diffuse. By \cite[Theorem 1.1]{IPP05}, we have $Q' \cap M \subset M_1$. Denote by $z \in \mathcal Z(Q' \cap M)$ the maximum projection such that $Qz \subset z M_1 z$. We prove that $z = 1$. Assume by contradiction that this not the case and put $q = z^\perp = 1 - z \in \mathcal Z(Q' \cap M) \subset M_1$. We have $q \neq 0$. 

First, assume that $Qq$ is amenable. Choose a diffuse abelian subalgebra $A \subset q^\perp M_1 q^\perp$ and put $\mathcal Q = Qq \oplus A$. Then $\mathcal Q$ is amenable and $\mathcal Q \cap M_1$ is diffuse. Theorem \ref{key-thm} implies that the inclusion $M_1 \subset M$ has the asymptotic orthogonality property relative to the diffuse subalgebra $\mathcal Q \cap M_1$ in the sense of \cite[Definition 5.1]{Ho12b}. Since the inclusion $M_1 \subset M$ is mixing (see e.g.\ \cite[Proposition 4.7]{Ho12b}) in the sense of \cite[Definition 4.4]{Ho12b}, we have that the inclusion $M_1 \subset M$ is weakly mixing through the diffuse subalgebra $\mathcal Q \cap M_1$ in the sense of \cite[Definition 4.1]{Ho12b}. Therefore \cite[Theorem 8.1]{Ho12b} implies that $\mathcal Q \subset M_1$ and so $Qq \subset q M_1 q$. This contradicts the fact that $z$ is the maximum projection in $\mathcal Z(Q' \cap M)$ such that $Qz \subset z M_1 z$.

Second, assume that $Qq$ is not amenable. Let $q_0 \in \mathcal Z(Q' \cap M)q$ be a nonzero central projection such that $Qqq_0$ has no amenable direct summand. Since $(Qqq_0)' \cap (qq_0 M qq_0)^\omega = qq_0(Q' \cap M^\omega)qq_0$ is diffuse and since the inclusion $M_1 \subset M$ is mixing, by \cite[Theorems 4.3, 4.5 and Lemma 5.1]{Pe06} and \cite[Theorem 4.3]{IPP05} (see also \cite[Theorem 5.6]{Ho07} and \cite[Theorem 6.3]{Io12}), we obtain that $Qqq_0 \preceq_M M_i$ for some $i \in \{1, 2\}$. This implies that $Qq \preceq_M M_i$ for some $i \in \{1, 2\}$.

There exist $\ell \geq 1$, a projection $p \in \mathbf M_\ell(M_i)$, a nonzero partial isometry $v \in \mathbf M_{1,\ell} (q M)p$ and a unital normal $\ast$-homomorphism $\varphi : Qq \to p\mathbf M_\ell(M_i)p$ such that $a v = v \varphi(a)$ for all $a \in Qq$. Write $v = [v_1 \cdots v_\ell] \in \mathbf M_{1,\ell} (q M)p$. In particular, we have $Q v_j \subset \sum_{k= 1}^\ell v_k M_i$ for all $1 \leq j \leq \ell$ and so $(Q \cap M_1) v_j \subset \sum_{k = 1}^\ell v_k M_i$ for all $1 \leq j \leq \ell$. Since $Q \cap M_1$ is diffuse, by \cite[Theorem 1.1]{IPP05}, we obtain that $i = 1$ and that $v_j \in M_1$ for all $1 \leq j \leq \ell$. Therefore $vv^* \in (Qq)' \cap q M_1 q$ is a nonzero projection such that $Q vv^* \subset vv^* M_1 vv^*$. If we denote by $z_0$ the central support of $vv^*$ in $(Qq)' \cap q M_1 q$, we have that $z_0 \in \mathcal Z(Q' \cap M)q$, $z_0 \neq 0$ and $Q z_0 \subset z_0 M_1 z_0$. This contradicts again the fact that $z$ is the maximum projection in $\mathcal Z(Q' \cap M)$ such that $Qz \subset z M_1 z$. 

Consequently, we obtain that $z = 1$ and so $Q \subset M_1$. This finishes the proof of Theorem \ref{step1}.
\end{proof}

\begin{proof}[Proof of Theorem \ref{thmD}]
The proof is similar to the one of Corollary \ref{corB}. Let $Q \subset M$ be any von Neumann subalgebra such that $Q \cap M_1$ is diffuse. By \cite[Lemma 2.7]{Io12}, there exists a central projection $z \in \mathcal Z(Q' \cap M) \cap \mathcal Z(Q' \cap M^\omega) \subset M_1$ such that $(Q' \cap M^\omega)z$ is diffuse and $(Q' \cap M^\omega)z = (Q' \cap M)z$ is discrete. Choose a diffuse abelian subalgebra $A \subset z^\perp M_1 z^\perp$ and put $\mathcal Q = Qz \oplus A$. We have that $\mathcal Q \cap M_1$ and $\mathcal Q' \cap M^\omega$ are diffuse. By Theorem \ref{step1}, we obtain $\mathcal Q \subset M_1$ and hence $Qz \subset z M_1 z$. This finishes the proof of Theorem \ref{thmD}.
\end{proof}

\end{document}